\documentclass[a4paper,12pt]{article}
\usepackage{amsmath}
\usepackage{amssymb}
\usepackage{amsthm}
\usepackage{amsfonts}
\usepackage{mathrsfs}
\usepackage{graphicx}
\usepackage{hyperref}
\usepackage{float}
\usepackage{enumitem}
\usepackage{slashed}
\usepackage{cite}
\usepackage{xcolor}
\usepackage{authblk}
\usepackage[none]{hyphenat}

\newtheorem{lemma}{Lemma}

\newtheorem{preposition}{Preposition}
\newtheorem{assumption}{Assumption}

\numberwithin{equation}{section}

\newtheoremstyle{named}{}{}{\itshape}{}{\bfseries}{.}{.5em}{\thmnote{#3 }#1}
\theoremstyle{named}
\newtheorem*{namedtheorem}{Theorem}

\hypersetup{ pdftitle={MCSH},
	pdfauthor={Mulyanto}, 
	pdfkeywords={}, 
	bookmarksnumbered, pdfstartview={FitH}, urlcolor=blue, }

\begin{document}
	\pagestyle{plain}

\title{\LARGE\textbf{The Global Existence and Uniqueness of Maxwell-Chern-Simons-Higgs Equation in (2+1) Dimensions}}

\author[1]{Mulyanto\footnote{Corresponding author}}
\author[1]{Ardian N. Atmaja}
\author[2]{Fiki T. Akbar}
\author[2]{Bobby E. Gunara}

\affil[1]{\textit{\small Research Center for Quantum Physics, National Research and Innovation Agency (BRIN),
Kompleks PUSPIPTEK Serpong, Tangerang 15310, Indonesia}}
\affil[2]{\textit{\small Theoretical Physics Laboratory,
	Theoretical High Energy Physics Research Division,
	Faculty of Mathematics and Natural Sciences,
	Institut Teknologi Bandung
	Jl. Ganesha no. 10 Bandung, Indonesia, 40132}}

\date{\small email: muly031@brin.go.id, ardian\textunderscore n\textunderscore a@yahoo.com, ftakbar@itb.ac.id, bobby@itb.ac.id}

\maketitle

\begin{abstract}
In this paper, we show the global existence and uniqueness of classical solutions of the Maxwell-Chern-Simmons-Higgs system coupled to a neutral scalar with nontrivial scalar potential on (2+1) dimensional Minkowski spacetime. Our methods rely only on classical existence theorems, including energy estimates, the Sobolev inequality, and the choice of the Coulomb gauge condition. The equations are well-posed for finite initial data and the solution preserves any additional $H^{s}$ regularity for $s>0$ in the data. 
\end{abstract}

\section{Introduction}
In this paper we will construct the proof of the global existence and uniqueness of the classical solutions of the Maxwell-Chern-Simons-Higgs (MCSH) system coupled to a neutral scalar with scalar potential on $(2+1)$ dimensional Minkowski spacetime  with metric $ \eta_{\mu\nu}={\text{diag}}\left(-,+,+\right)$ which is described by a Lagrangian density of the form \cite{kim}: 
\begin{eqnarray}\label{lagrange}
  \mathcal{L}=-\frac{1}{4}{{F}_{\mu \nu }}{{F}^{\mu \nu }}-\frac{\kappa }{4}{{\epsilon }^{\mu \rho \sigma }}{{A}_{\mu }}{{F}_{\rho \sigma }}+\frac{1}{2}{{\partial }_{\mu }}N{{\partial }^{\mu }}N+|{{D}_{\mu }}\phi {{|}^{2}}-V(|\phi |, N)~.
\end{eqnarray}
Here, the electromagnetic field strength tensor is defined as $F_{\mu\nu} \equiv \partial_\mu A_\nu - \partial_\nu A_\mu $, $N$ represents the real scalar field, $\kappa>0$ is the Chern-Simons constant, and the covariant derivative of the complex Higgs field is given by $ D_\mu \phi \equiv \partial_\mu \phi + i A_\mu \phi $. We adopt the Einstein summation convention with Greek indices $\mu, \nu, \rho, \sigma$ ranging over $0, 1, 2$ and Latin indices $i,j,k$ running over $1, 2$. To ensure there is no confusion about index notation in the following discussion, we also use the indices $m$ and $q$ to indicate the degree of a polynomial as in \eqref{V4} and $\Tilde{k}$ is the $k$-th order of the derivative with respect to spatial coordinate as in \eqref{J}. In this analysis, the potential $ V(|\phi|,N) $ was initially considered as an arbitrary function. The fields equations of motion derived from Lagrangian \eqref{lagrange} are
 \begin{eqnarray}\label{eom1}
     {{\partial }_{\mu }}{F}^{\mu \nu }=-\operatorname{Im}\left( \phi \overline{{{D}^{\nu }}\phi } \right)+\kappa {{F}^{\nu }}~,
 \end{eqnarray}
 \begin{eqnarray}\label{eom2}
     D_{\mu}D^{\mu}\phi={{\partial }_{\phi }}V~,
 \end{eqnarray}
 \begin{eqnarray}\label{eom3}
     {{\partial }_{\mu }}{{\partial }^{\mu }}N+{{\partial }_{N}}V=0.
 \end{eqnarray}
where ${{F}^{\nu }}=\frac{1}{2}{{\epsilon }^{\nu \rho \sigma }}{{F}_{\rho \sigma }}$ is the dual electromagnetic vector. 

The MCSH model is one of important topics in planar physics that could provide the study of particles that possess magnetic and electric charges.  In this area, the study of the self-dual property of the system in (2 + 1)-dimensional gauge field theory is crucial since it provides a system of first-order equations that minimizes static energy \cite{han2}. These equations are also known as the BPS equations \cite{bogo, prasad}. The most basic models that have a self-dual structure are the Maxwell-Higgs (MH) model \cite{casana, beze2, beze3} and the Chern-Simons-Higgs (CSH) model \cite{jack, hong, dos, dos2, dos3}. On the other hand, one might consider combining the Maxwell field with the Chern-Simons interactions into a single model. However, this combination lacked a self-dual structure. This issue was addressed in \cite{Lee} by merging these interactions using a supersymmetry-based approach, aided by the introduction of an additional scalar field $N$, known as the neutral scalar field. In this way, they introduced the Maxwell-Chern-Simons-Higgs (MCSH) model, which implies a self-dual structure \cite{han}. Recent studies also show that the introduction of non-minimal couplings leads to a generalized Maxwell-Chern-Simons-Higgs (GMCSH) model \cite{andrade, bezeia, yuda}.

For the Cauchy problem of the MCSH model in (2+1)-Minkowski spacetime, the global existence of the system in the Lorentz gauge condition has been proven in \cite{chae}. The method is based on the classical result of \cite{Moncrief} on MH fields by demonstrating that an appropriately defined higher-order energy, though not strictly conserved, does not blow up in finite time. This method is also used for more complex systems such as the Yang-Mills-Higgs system \cite{edrey}. However, this method does not necessarily provide global regularity of smooth solutions to these equations. This problem, especially for Maxwell-Klein-Gordon cases, was solved by \cite{Klainerman} using a new estimate for a nonlinear wave equation that relies on the Coulomb gauge condition. As for the MCSH system, the global existence and uniqueness problem for $H^{s}$ with $s>0$ has to be solved. However, in contrast to \cite{Klainerman}, the MCSH system contains a Chern-Simons term, which introduces additional gauge terms. This term significantly changes the structure of the equations and requires a different approach to handle nonlinearities and their estimates.

Since all of the methods above provide gauge conditions, it is fundamental to show that the proposed system remains invariant under the gauge transformations $U(1)$.
\begin{eqnarray}
    \phi\to e^{i\chi}\phi \hspace{2em} A_\mu\to A_\mu-\partial_\mu\chi~.
\end{eqnarray}
We can replace $A_\mu$ with $ A_\mu-\partial_\mu\chi$ and derive the same $F_{\mu\nu}$. Correspondingly, we replace ${{D}_{\mu }}={{\partial }_{\mu }}+i{{A}_{\mu }}$ with $e^{i\chi}D_\mu e^{-i\chi}={{\partial }_{\mu }}+i\left( {{A}_{\mu }}-{{\partial }_{\mu }}\chi  \right)$. Therefore, if we replace $A_\mu$ with $ A_\mu-\partial_\mu\chi$ and $\phi$ with $e^{-i\chi}\phi$ we can obtain the same $D_\mu\phi$. All that remains is to check the invariance of the Chern-Simons term. This can be done by the assumption that, on a non-compact manifold like Minkowski spacetimes, the field at the boundary, $x\to\infty$, can be regarded as vanishing. Consequently, a key aspect of our results is the selection of a gauge condition, specifically the Coulomb gauge. More precisely, we construct solutions to the system \eqref{eom1}-\eqref{eom2} that satisfy the gauge condition 
\begin{eqnarray}\label{CG}
    \partial^i A_i=0~.
\end{eqnarray}
By coupling \eqref{CG} with \eqref{eom1} and \eqref{eom2}, we obtain
\begin{eqnarray}\label{eomCG1}
    \Delta {{A}^{0}}= \kappa {{F}^{0}}-\operatorname{Im}\left( \phi \overline{{{D}^{0}}\phi } \right)~,
\end{eqnarray}
\begin{eqnarray}\label{eomCG2}
    \Box{{A}^{i}}= \kappa {{F}^{i}}-\operatorname{Im}\left( \phi \overline{{{D}^{i}}\phi } \right) +{{\partial }^{i}}{{\partial }_{0}}{{A}^{0}}~,
\end{eqnarray}
\begin{eqnarray}\label{eomCG3}
   \Box\phi &=&-{{\partial }_{\phi }}V -2i{{A}_{\mu }}{{D}^{\mu }}\phi-i{{\partial }_{\mu }}\left( {{A}^{\mu }}\phi  \right) ~. 
\end{eqnarray}
Here, we also split the gauge potential $A_\mu$ into its time component $A_0$ and its spatial component $A = \left( A_i \right)_{i=1,2}$. For a function $\phi(t,x)$, we use $\nabla \phi = \left( \partial_i \phi \right)_{i=1,2}$ to denote its spatial gradient and $\partial \phi = \left( \partial_0, \nabla \phi \right)$ to represent its full space-time gradient. The D’Alembertian operator, denoted by $\Box$, is defined as $\Box = -\partial_t^2 + \Delta = -\partial_t^2 + \partial_1^2 + \partial_2^2$. Given that $\partial^i A_i = 0$, we find that $\nabla \times (\nabla \times A) = -\Delta A$, where $\nabla \times A$ represents the standard curl operation for vector fields in $\mathbb{R}^2$. Explicitly, $(\nabla \times A)^i = \epsilon^{ijk} \partial_j A_k$. Let $\mathcal{P}$ be the projection operator onto divergence-free vector fields. For any vector field $B$, we define it by
\begin{eqnarray}
    \mathcal{P} B \equiv \Delta^{-1} (\nabla \times (\nabla \times B))
\end{eqnarray}
It’s clear that if $B$ is divergence-free ($\partial^i B_i = 0$), then $\mathcal{P} B = -B$. With this in mind, equation \eqref{eomCG2} implies the following
\begin{eqnarray}\label{eomCG2P}
    \Box{{A}^{i}}=\mathcal{P}\left\{ \kappa {{F}^{i}}-\operatorname{Im}\left( \phi \overline{{{D}^{i}}\phi } \right) \right\}~.
\end{eqnarray}
Remark that equation \eqref{CG} always holds for all time if the initial data $A_i(0, \cdot)$ and $\partial_t A_i(0, \cdot)$ are divergence-free. 

It is believed that under the Coulomb gauge condition, the MCSH system \eqref{eom3}, \eqref{eomCG1}, \eqref{eomCG3} and \eqref{eomCG2P} also ensure global regularity solutions. In this study, we demonstrate that by applying the method from \cite{Klainerman}, the MCSH system admits a globally unique solution while preserving $H^{s}$ regularity with $s>0$. Following this method, we define the functional energy required for the estimates and state our main theorem in section \ref{sec:MCSH energy}. In section \ref{sec:MCSH estimate}, we demonstrate the estimates for $A_0, A, N$ and $\phi$ that use to prove the global existence and uniqueness for the MCSH systems. Finally, we prove the main theorem of this work in Section \ref{sec:MCSH prove}.

\section{The Functional Energy}\label{sec:MCSH energy}

The electromagnetic field $F_{\mu\nu}$ can be decomposed into its electric and magnetic components,
 \begin{eqnarray}
     E_i=F_{0i}\hspace{2em} B_i=\Tilde{F}_{0i}=\frac{1}{2}\varepsilon_{ijk}F^{jk}~.
 \end{eqnarray}
 The total energy at a given time is given by the expression
\begin{eqnarray}
    {\mathcal{E}}(t)=\int\limits_{{{\mathbb{R}}^{2}}}{{{T}^{00}}dx}=\int\limits_{{{\mathbb{R}}^{2}}}{\left\{ \frac{1}{2}\left( {{E}_{i}}^{2}+H_{i}^{2} \right)+\left|\partial_0N\right|^2+\frac{1}{2}\left( {{\left| {{D}_{0}}\phi  \right|}^{2}}+{{\left| {{D}_{i}}\phi  \right|}^{2}} \right)+V \right\}dx}~.
\end{eqnarray}
We analyze the initial value problem at $ t = 0 $ where the initial data $ E(0, \cdot) $, $ H(0, \cdot) $, $\phi(0, \cdot)$ and $\partial_t \phi(0, \cdot) $ are provided. Specifically, we examined the time evolution of the system described by \eqref{eom1}–\eqref{eom2} for data with finite energy. In other words, we assume that the data satisfies the condition
\begin{eqnarray}
    \mathcal{E}(0)<\infty~.
\end{eqnarray}
In particular, we have
\begin{eqnarray}\label{A}
   {{\left\| {{\partial }_{i}}{{A}_{\mu }} \right\|}_{{{L}^{2}}\left( {{\mathbb{R}}^{2}} \right)}}\le {{\left\| E \right\|}_{{{L}^{2}}\left( {{\mathbb{R}}^{2}} \right)}}+{{\left\| H \right\|}_{{{L}^{2}}\left( {{\mathbb{R}}^{2}} \right)}}~.
\end{eqnarray}
In view of \eqref{A}, we introduce a functional energy norm 
\begin{eqnarray}\label{J}
  \mathcal{J}\left( t \right)& =& {{\left\| {{\partial }_{\mu }}{{A}_{\nu }} \right\|}_{{{L}^{2}}\left( {{\mathbb{R}}^{2}} \right)}}+{{\left\| \phi  \right\|}_{{{L}^{2}}\left( {{\mathbb{R}}^{2}} \right)}}+{{\left\| {{\partial }_{\mu }}\phi  \right\|}_{{{L}^{2}}\left( {{\mathbb{R}}^{2}} \right)}}+{{\left\| {{\partial }_{\mu }}\phi  \right\|}_{{{L}^{2}}\left( {{\mathbb{R}}^{2}} \right)}}+{{\left\| N  \right\|}_{{{L}^{2}}\left( {{\mathbb{R}}^{2}} \right)}} \notag\\ 
 & & +{{\left\| {{\partial }_{\mu }}N  \right\|}_{{{L}^{2}}\left( {{\mathbb{R}}^{2}} \right)}}+\sum\limits_{\Tilde{k}=2}^{4}{{{\left\| {{\partial_i}^{\Tilde{k}}}A \right\|}_{{{L}^{2}}\left( {{\mathbb{R}}^{2}} \right)}}}+\sum\limits_{\Tilde{k}=2}^{4}{{{\left\| {{\partial_i }^{\Tilde{k}}}\phi  \right\|}_{{{L}^{2}}\left( {{\mathbb{R}}^{2}} \right)}}}~.
\end{eqnarray}
where $\Tilde{k}$ is the $k$-th order of the derivative. By applying the law of conservation of total energy, we can demonstrate that functional energy $\mathcal{J}(t) $  is bounded for all time $t>0$. In fact, we have the following preposition.
\begin{preposition}\label{prepos1}
Let $\phi, A_0, A$ be a smooth solution of the MCSH system as in equations \eqref{eom1}-\eqref{eom2} under the Coulomb gauge condition \eqref{CG}. Suppose that the initial energy $ \mathcal{J}_0 = \mathcal{J}(0) $ is finite:
    \begin{eqnarray}
        \mathcal{J}_0<\infty~.
    \end{eqnarray}
Then, for all $t\geq 0$, there exists a positive constant $C$, dependent only on $ \mathcal{J}_0 $, such that the following inequality holds
\begin{eqnarray}
       \mathcal{J}(t)\le C\left(1+t\right)^2~.
\end{eqnarray}
\end{preposition}
\begin{proof}
Clearly that ${{\left\| {{\partial }_{\mu }}N  \right\|}_{{{L}^{2}}\left( {{\mathbb{R}}^{2}} \right)}}\le\mathcal{J}_0$. From equation \eqref{A}, we can also demonstrate that
\begin{eqnarray}\label{L2derifN}
  {{\left\| {{\partial }_{\mu }}{{A}_{\nu }} \right\|}_{{{L}^{2}}}}&\le& {{\left\| E\left( t,\cdot  \right) \right\|}_{{{L}^{2}}}}+{{\left\| H\left( t,\cdot  \right) \right\|}_{{{L}^{2}}}} \notag\\ 
 & \le& {{\mathcal{E}}^{1/2}} \notag\\ 
 & \le& {{\mathcal{J}}_{0}}  ~.
\end{eqnarray}
To estimate the second term of \eqref{J}, let us consider
\begin{eqnarray}
   \frac{d}{dt}\left\| \phi  \right\|_{{{L}^{2}}\left( {{\mathbb{R}}^{2}} \right)}^{2}&\le& 2\int\limits_{{{\mathbb{R}}^{n}}}{\Re \left[ \phi \overline{{{D}_{0}}\phi } \right]}dx \notag\\ 
 & \le& {{\left\| \phi  \right\|}_{{{L}^{2}}\left( {{\mathbb{R}}^{2}} \right)}}{{\left\| {{D}_{0}}\phi  \right\|}_{{{L}^{2}}\left( {{\mathbb{R}}^{2}} \right)}} \notag\\ 
 & \le& {{\mathcal{E}}^{1/2}}{{\left\| \phi  \right\|}_{{{L}^{2}}\left( {{\mathbb{R}}^{2}} \right)}}~.  
\end{eqnarray}
Consequently, we have
\begin{eqnarray}\label{L2phi}
 {{\left\| \phi  \right\|}_{{{L}^{2}}\left( {{\mathbb{R}}^{2}} \right)}}\le {{\mathcal{J}}_{0}}\left( 1+t \right)~.
\end{eqnarray}
In the same way, we also have
\begin{eqnarray}\label{L2derifA}
 {{\left\| A \right\|}_{{{L}^{2}}\left( {{\mathbb{R}}^{2}} \right)}}\le {{\mathcal{J}}_{0}}\left( 1+t \right)~,
\end{eqnarray}
\begin{eqnarray}\label{L2N}
 {{\left\| N \right\|}_{{{L}^{2}}\left( {{\mathbb{R}}^{2}} \right)}}\le {{\mathcal{J}}_{0}}\left( 1+t \right)~.
\end{eqnarray}
To get the estimate for ${\left\|\nabla \phi \right\|}_{{L}^{2}}$, we use the Minkowski inequality along with H\"older inequality and Gagliardo-Nirenberg-Sobolev inequality to obtain
\begin{eqnarray}
   {{\left\| \nabla \phi  \right\|}_{{{L}^{2}}}}&\le& {{\left\| {{D}_{i}}\phi  \right\|}_{{{L}^{2}}}}+{{\left\| {{A}_{i}}\phi  \right\|}_{{{L}^{2}}}} \notag\\ 
 & \le& {{\mathcal{J}}_{0}}+{{\left\| {{A}_{i}} \right\|}_{{{L}^{6}}}}{{\left\| \phi  \right\|}_{{{L}^{3}}}} \notag\\ 
 & \le & {{\mathcal{J}}_{0}}+\left\| \nabla {{A}_{i}} \right\|_{{{L}^{2}}}^{2/3}\left\| {{A}_{i}} \right\|_{{{L}^{2}}}^{1/3}\left\| \phi  \right\|_{{{L}^{2}}}^{2/3}\left\| \nabla \phi  \right\|_{{{L}^{2}}}^{1/3} \notag\\ 
 & \le & {{\mathcal{J}}_{0}}+{{\mathcal{J}}_{0}}^{5/3}\left( 1+t \right)\left\| \nabla \phi  \right\|_{{{L}^{2}}}^{1/3} ~, 
\end{eqnarray}
which lead to
\begin{eqnarray}\label{L2defphi}
    {{\left\| \nabla\phi  \right\|}_{{{L}^{2}}}}\le C(\mathcal{J}_0)~.
\end{eqnarray}
The same procedure also apply for
\begin{eqnarray}
   {{\left\| {{\partial }_{0}}\phi  \right\|}_{{{L}^{2}}}}&\le& {{\left\| {{D}_{0}}\phi  \right\|}_{{{L}^{2}}}}+{{\left\| {{A}_{0}}\phi  \right\|}_{{{L}^{2}}}} \notag\\ 
 & \le& {{\mathcal{J}}_{0}}+{{\left\| {{A}_{0}} \right\|}_{{{L}^{6}}}}{{\left\| \phi  \right\|}_{{{L}^{3}}}} \notag\\ 
 & \le& {{\mathcal{J}}_{0}}+\left\| \nabla {{A}_{0}} \right\|_{{{L}^{2}}}^{2/3}\left\| {{A}_{0}} \right\|_{{{L}^{2}}}^{1/3}\left\| \phi  \right\|_{{{L}^{2}}}^{2/3}\left\| \nabla \phi  \right\|_{{{L}^{2}}}^{1/3} \notag\\ 
 & \le& {{\mathcal{J}}_{0}}+{{\mathcal{J}}_{0}}^{5/3}\left( 1+t \right)\left\| \nabla \phi  \right\|_{{{L}^{2}}}^{1/3} \notag\\ 
 & \le& {{\mathcal{J}}_{0}}+{{\mathcal{J}}_{0}}^{2}{{\left( 1+t \right)}} \notag\\ 
 & \le& C{{\left( 1+t \right)}} ~. 
\end{eqnarray}
To get estimate for the last two term of \eqref{J}, we have to specify the form of the scalar potential $V(|\phi |, N)$. In this research, we consider the following potential:
\begin{assumption}\label{asumption}
The scalar potential  $V(|\phi |, N)$ in \eqref{lagrange} has to be  of the following form
\begin{eqnarray}
   \label{V4} V\left( \left| \phi  \right|,N \right) = \sum\limits_{m=1}^{M} \sum\limits_{q=1}^{Q} \alpha_{mq} {{\left| \phi  \right|}^{2m}} {{N}^{q}}
\end{eqnarray}
where $\alpha_{mq}$ are real constants. 
\end{assumption}
Hence, we have the following lemma
\begin{lemma}\label{lemmaV}
    Lets $V(|\phi |, N)$ satisfy condition \eqref{V4}, Then, the following inequality holds:
    \begin{eqnarray}
        \label{v41}{{\left\| {{\partial }_{\phi }}V \right\|}_{{{L}^{2}}\left( {{\mathbb{R}}^{2}} \right)}}\le C(1+t)\\
       \label{v42} {{\left\| \nabla{{\partial }_{\phi }}V \right\|}_{{{L}^{2}}\left( {{\mathbb{R}}^{2}} \right)}}\le C(1+t)\\
       \label{v43} {{\left\| \nabla^2{{\partial }_{\phi }}V \right\|}_{{{L}^{2}}\left( {{\mathbb{R}}^{2}} \right)}}\le C(1+t)^2
    \end{eqnarray}
\end{lemma}
\begin{proof}
   By using the H\"older and Gagliardo-Nirenberg-Sobolev inequality, we have
   \begin{eqnarray}
   {{\left\| {{\partial }_{\phi }}V \right\|}_{{{L}^{2}}}}&\le& c \sum\limits_{m=1}^{M}{\sum\limits_{q=1}^{Q}{{{\left\| {{\phi }^{2m-1}}{{N}^{q}} \right\|}_{{{L}^{2}}}}}} \notag\\ 
 & \le&c \sum\limits_{m=1}^{M}{\sum\limits_{q=1}^{Q}{{{\left\| {{\phi }^{2m-1}} \right\|}_{{{L}^{4}}}}}}{{\left\| {{N}^{q}} \right\|}_{{{L}^{4}}}}\notag\\ 
 & \le&c \sum\limits_{m=1}^{M}{\sum\limits_{q=1}^{Q}{\left\| \phi  \right\|_{{{L}^{4\left( 2m-1 \right)}}}^{\left( 2m-1 \right)}\left\| N \right\|_{{{L}^{4q}}}^{q}}} \notag\\ 
 & \le&c \sum\limits_{m=1}^{M}{\sum\limits_{q=1}^{Q}{\left\| \nabla \phi  \right\|_{{{L}^{2}}}^{\frac{4m-3}{2}}\left\| \nabla N \right\|_{{{L}^{2}}}^{\frac{2q-1}{2}}\left\| \phi  \right\|_{{{L}^{2}}}^{\frac{1}{2}}\left\| N \right\|_{{{L}^{2}}}^{\frac{1}{2}}}} \notag\\ 
 & \le& C\left( 1+t \right) ~, 
   \end{eqnarray}
   \begin{eqnarray}
   {{\left\| {{\partial }_{N}}V \right\|}_{{{L}^{2}}}}&\le& c\sum\limits_{m=1}^{M}{\sum\limits_{q=1}^{Q}{{{\left\| {{\phi }^{2m}}{{N}^{q-1}} \right\|}_{{{L}^{2}}}}}} \notag\\ 
 & \le& c\sum\limits_{m=1}^{M}{\sum\limits_{q=1}^{Q}{{{\left\| {{N}^{q-1}} \right\|}_{{{L}^{4}}}}{{\left\| {{\phi }^{2m}} \right\|}_{{{L}^{4}}}}}}\notag\\ 
 & \le& c\sum\limits_{m=1}^{M}{\sum\limits_{q=1}^{Q}{\left\| N \right\|_{{{L}^{4\left( 2q-1 \right)}}}^{\left( 2q-1 \right)}\left\| \phi  \right\|_{{{L}^{8m}}}^{2m}}} \notag\\ 
 & \le& c\sum\limits_{m=1}^{M}{\sum\limits_{q=1}^{Q}{\left\| \nabla \phi  \right\|_{{{L}^{2}}}^{\frac{4m-1}{2}}\left\| \phi  \right\|_{{{L}^{2}}}^{\frac{1}{2}}\left\| N \right\|_{{{L}^{2}}}^{\frac{1}{2}}\left\| \nabla N \right\|_{{{L}^{2}}}^{\frac{4q-3}{2}}}} \notag\\ 
 & \le& C\left( 1+t \right)~, 
   \end{eqnarray}
where we have set $c \equiv \sum\limits_{m=1}^{M}{\sum\limits_{q=1}^{Q}|\alpha_{mq}|}$. To prove \eqref{v42}, first we need to estimate for $\Tilde{k}=2$, we use the equation of motion \eqref{eom1} and \eqref{eom2}, hence
\begin{eqnarray}
   {{\left\| {{\nabla }^{2}}A \right\|}_{{{L}^{2}}\left( {{\mathbb{R}}^{2}} \right)}}&\le& {{\left\| {{A}^{2}}\phi  \right\|}_{{{L}^{2}}}}+{{\left\| \nabla A \right\|}_{{{L}^{2}}}} \notag\\ 
 & \le& {{\left\| {{A}^{2}} \right\|}_{{{L}^{4}}}}{{\left\| \phi  \right\|}_{{{L}^{4}}}}+{{\left\| \nabla A \right\|}_{{{L}^{2}}}} \notag\\ 
 & \le& \left\| A \right\|_{{{L}^{8}}}^{2}\left\| \nabla \phi  \right\|_{{{L}^{2}}}^{1/2}\left\| \phi  \right\|_{{{L}^{2}}}^{1/2}+{{\left\| \nabla A \right\|}_{{{L}^{2}}}} \notag\\ 
 & \le& \left\| \nabla A \right\|_{{{L}^{2}}}^{3/2}\left\| A \right\|_{{{L}^{2}}}^{1/2}\left\| \nabla \phi  \right\|_{{{L}^{2}}}^{1/2}\left\| \phi  \right\|_{{{L}^{2}}}^{1/2}+{{\left\| \nabla A \right\|}_{{{L}^{2}}}} \notag\\ 
 & \le& {{\mathcal{J}}_{0}}^{2}\left( 1+t \right)+{{\mathcal{J}}_{0}} \notag\\ 
 & \le& C\left( 1+t \right)~,  
\end{eqnarray}
\begin{eqnarray}
   {{\left\| {{\nabla }^{2}}\phi  \right\|}_{{{L}^{2}}}}&\le& {{\left\| {{\partial }_{\phi }}V \right\|}_{{{L}^{2}}}}+{{\left\| \nabla A\phi  \right\|}_{{{L}^{2}}}}+{{\left\| A\nabla \phi  \right\|}_{{{L}^{2}}}}+{{\left\| {{A}^{2}}\phi  \right\|}_{{{L}^{2}}}} \notag\\ 
 & \le& {{\left\| {{\partial }_{\phi }}V \right\|}_{{{L}^{2}}}}+{{\left\| \nabla A \right\|}_{{{L}^{4}}}}{{\left\| \phi  \right\|}_{{{L}^{4}}}}+{{\left\| A \right\|}_{{{L}^{4}}}}{{\left\| \nabla \phi  \right\|}_{{{L}^{4}}}}+\left\| A \right\|_{{{L}^{8}}}^{2}{{\left\| \phi  \right\|}_{{{L}^{4}}}} \notag\\ 
 & \le& {{\left\| {{\partial }_{\phi }}V \right\|}_{{{L}^{2}}}}+\left\| {{\nabla }^{2}}A \right\|_{{{L}^{2}}}^{1/2}\left\| \nabla A \right\|_{{{L}^{2}}}^{1/2}\left\| \nabla \phi  \right\|_{{{L}^{2}}}^{1/2}\left\| \phi  \right\|_{{{L}^{2}}}^{1/2} \notag\\ 
 & &+\left\| \nabla A \right\|_{{{L}^{2}}}^{1/2}\left\| A \right\|_{{{L}^{2}}}^{1/2}\left\| {{\nabla }^{2}}\phi  \right\|_{{{L}^{2}}}^{1/2}\left\| \nabla \phi  \right\|_{{{L}^{2}}}^{1/2} \notag\\ 
 && +\left\| \nabla A \right\|_{{{L}^{2}}}^{3/2}\left\| A \right\|_{{{L}^{2}}}^{1/2}\left\| \nabla \phi  \right\|_{{{L}^{2}}}^{1/2}\left\| \phi  \right\|_{{{L}^{2}}}^{1/2} \notag\\ 
 & \le& C\left( 1+t \right)~.  
\end{eqnarray}
Therefore,
\begin{eqnarray}
   {{\left\| \nabla \left( {{\partial }_{\phi }}V \right) \right\|}_{{{L}^{2}}}}& \le& \sum\limits_{m=1}^{M}{\sum\limits_{q=1}^{Q}{\left( {{\left\| {{\phi }^{2m-2}}{{N}^{q}}\nabla \phi  \right\|}_{{{L}^{2}}}}+{{\left\| {{\phi }^{2m-1}}{{N}^{q-1}}\nabla N \right\|}_{{{L}^{2}}}} \right)}}\notag\\
  &\le& \left\| \nabla N \right\|_{{{L}^{2}}}^{\frac{1}{2}}\left\| N \right\|_{{{L}^{2}}}^{\frac{1}{2}}\left\| \nabla \phi  \right\|_{{{L}^{2}}}^{1/2}\left\| {{\nabla }^{2}}\phi  \right\|_{{{L}^{2}}}^{1/2} \notag\\ 
 && +\sum\limits_{m=2}^{M}{\sum\limits_{q=2}^{Q}{\left\| \nabla N \right\|_{{{L}^{2}}}^{\frac{4q-1}{4}}\left\| N \right\|_{{{L}^{2}}}^{\frac{1}{4}}\left\| \nabla \phi  \right\|_{{{L}^{2}}}^{\frac{8m-9}{4}}\left\| \phi  \right\|_{{{L}^{2}}}^{\frac{1}{4}}\left\| \nabla \phi  \right\|_{{{L}^{2}}}^{1/2}\left\| {{\nabla }^{2}}\phi  \right\|_{{{L}^{2}}}^{1/2}}} \notag\\
  && +\left\| \nabla \phi  \right\|_{{{L}^{2}}}^{1/2}\left\| \phi  \right\|_{{{L}^{2}}}^{1/2}\left\| \nabla N \right\|_{{{L}^{2}}}^{1/2}\left\| {{\nabla }^{2}}N \right\|_{{{L}^{2}}}^{1/2} \notag\\ 
 && +\sum\limits_{m=2}^{M}{\sum\limits_{q=2}^{Q}{\left\| \nabla \phi  \right\|_{{{L}^{2}}}^{\frac{8m-5}{4}}\left\| \phi  \right\|_{{{L}^{2}}}^{\frac{1}{4}}\left\| \nabla N \right\|_{{{L}^{2}}}^{\frac{4q-5}{4}}\left\| N \right\|_{{{L}^{2}}}^{\frac{1}{4}}\left\| \nabla N \right\|_{{{L}^{2}}}^{1/2}\left\| {{\nabla }^{2}}N \right\|_{{{L}^{2}}}^{1/2}}} \notag\\ 
 & \le&  C{\left( 1+t \right)}~.  
\end{eqnarray}
Next, for $\Tilde{k}=3$, we differentiate \eqref{eom1} and \eqref{eom2} to obtain,
\begin{eqnarray}
       {{\left\| {{\nabla }^{3}}A \right\|}_{{{L}^{2}}}}&\le& {{\left\| \nabla \left( {{A}^{2}}\phi  \right) \right\|}_{{{L}^{2}}}}+{{\left\| {{\nabla }^{2}}A \right\|}_{{{L}^{2}}}} \notag\\ 
 & \le& {{\left\| A\nabla A\phi  \right\|}_{{{L}^{2}}}}+{{\left\| {{A}^{2}}\nabla \phi  \right\|}_{{{L}^{2}}}}+{{\left\| {{\nabla }^{2}}A \right\|}_{{{L}^{2}}}} \notag\\ 
 & \le& {{\left\| \nabla A \right\|}_{{{L}^{4}}}}{{\left\| \phi  \right\|}_{{{L}^{8}}}}{{\left\| A \right\|}_{{{L}^{8}}}}+{{\left\| {{A}^{2}} \right\|}_{{{L}^{4}}}}{{\left\| \nabla \phi  \right\|}_{{{L}^{4}}}}+{{\left\| {{\nabla }^{2}}A \right\|}_{{{L}^{2}}}} \notag\\ 
 & \le& \left\| {{\nabla }^{2}}A \right\|_{{{L}^{2}}}^{1/2}\left\| \nabla A \right\|_{{{L}^{2}}}^{1/2}\left\| \nabla \phi  \right\|_{{{L}^{2}}}^{3/4}\left\| \phi  \right\|_{{{L}^{2}}}^{1/4}\left\| \nabla A \right\|_{{{L}^{2}}}^{3/4}\left\| A \right\|_{{{L}^{2}}}^{1/4} \notag\\ 
 & &+\left\| \nabla A \right\|_{{{L}^{2}}}^{3/2}\left\| A \right\|_{{{L}^{2}}}^{1/2}\left\| {{\nabla }^{2}}\phi  \right\|_{{{L}^{2}}}^{1/2}\left\| \nabla \phi  \right\|_{{{L}^{2}}}^{1/2}+{{\left\| {{\nabla }^{2}}A \right\|}_{{{L}^{2}}}} \notag\\ 
 & \le& C\left( 1+t \right)~,
\end{eqnarray}
\begin{eqnarray}
  {{\left\| {{\nabla }^{3}}\phi  \right\|}_{{{L}^{2}}}}& \le&  {{\left\| \nabla {{\partial }_{\phi }}V \right\|}_{{{L}^{2}}}}+{{\left\| \nabla^2 A\phi  \right\|}_{{{L}^{2}}}}+{{\left\| \nabla A\nabla \phi  \right\|}_{{{L}^{2}}}} \notag\\ 
 &&  +{{\left\| A\nabla^2 \phi  \right\|}_{{{L}^{2}}}}+{{\left\| A\nabla A\phi  \right\|}_{{{L}^{2}}}}+{{\left\| {{A}^{2}}\nabla \phi  \right\|}_{{{L}^{2}}}} \notag\\ 
 & \le&  {{\left\| \nabla {{\partial }_{\phi }}V \right\|}_{{{L}^{2}}}}+{{\left\| \nabla^2 A \right\|}_{{{L}^{4}}}}{{\left\| \phi  \right\|}_{{{L}^{4}}}}+{{\left\| \nabla A \right\|}_{{{L}^{4}}}}{{\left\| \nabla \phi  \right\|}_{{{L}^{4}}}} \notag\\ 
 &&  +{{\left\| \nabla^2 \phi  \right\|}_{{{L}^{4}}}}{{\left\| A \right\|}_{{{L}^{4}}}}+{{\left\| \nabla A \right\|}_{{{L}^{2}}}}{{\left\| A \right\|}_{{{L}^{8}}}}{{\left\| \phi  \right\|}_{{{L}^{8}}}}+{{\left\| {{A}^{2}} \right\|}_{{{L}^{4}}}}{{\left\| \nabla \phi  \right\|}_{{{L}^{4}}}} \notag\\ 
 & \le&  {{\left\| \nabla {{\partial }_{\phi }}V \right\|}_{{{L}^{2}}}}+\left\| {{\nabla }^{3}}A \right\|_{{{L}^{2}}}^{1/2}\left\| \nabla^2 A \right\|_{{{L}^{2}}}^{1/2}\left\| \nabla \phi  \right\|_{{{L}^{2}}}^{1/2}\left\| \phi  \right\|_{{{L}^{2}}}^{1/2} \notag\\ 
 &&  +\left\| \nabla^2 A \right\|_{{{L}^{2}}}^{1/2}\left\| \nabla A \right\|_{{{L}^{2}}}^{1/2}\left\| \nabla^2 \phi  \right\|_{{{L}^{2}}}^{1/2}\left\| \nabla \phi  \right\|_{{{L}^{2}}}^{1/2} \notag\\ 
 &&  +\left\| {{\nabla }^{3}}\phi  \right\|_{{{L}^{2}}}^{1/2}\left\| \nabla^2 \phi  \right\|_{{{L}^{2}}}^{1/2}\left\| \nabla A \right\|_{{{L}^{2}}}^{1/2}\left\| A \right\|_{{{L}^{2}}}^{1/2} \notag\\ 
 && +{{\left\| \nabla A \right\|}_{{{L}^{2}}}}\left\| \nabla A \right\|_{{{L}^{2}}}^{3/4}\left\| A \right\|_{{{L}^{2}}}^{1/4}\left\| \nabla \phi  \right\|_{{{L}^{2}}}^{3/4}\left\| \phi  \right\|_{{{L}^{2}}}^{1/4} \notag\\ 
 &&  +\left\| \nabla A \right\|_{{{L}^{2}}}^{3/2}\left\| A \right\|_{{{L}^{2}}}^{1/2}\left\| \nabla^2\phi  \right\|_{{{L}^{2}}}^{1/2}\left\| \nabla \phi  \right\|_{{{L}^{2}}}^{1/2} \notag\\ 
 & \le&  C{{\left( 1+t \right)}^{2}} ~. 
\end{eqnarray}
Applying the results above, we arrive at
\begin{eqnarray}
   {{\left\| {{\nabla }^{2}}\left( {{\partial }_{\phi }}V \right) \right\|}_{{{L}^{2}}}}&\le& \left\| \nabla N \right\|_{{{L}^{2}}}^{\frac{1}{2}}\left\| N \right\|_{{{L}^{2}}}^{\frac{1}{2}}\left\| {{\nabla }^{3}}\phi  \right\|_{{{L}^{2}}}^{1/2}\left\| {{\nabla }^{2}}\phi  \right\|_{{{L}^{2}}}^{1/2}+\left\| {{\nabla }^{2}}N \right\|_{{{L}^{2}}}^{\frac{1}{2}}\left\| \nabla N \right\|_{{{L}^{2}}}^{\frac{1}{2}}\notag\\ 
 && +\left\| \nabla \phi  \right\|_{{{L}^{2}}}^{1/2}\left\| {{\nabla }^{2}}\phi  \right\|_{{{L}^{2}}}^{1/2}+\left\| {{\nabla }^{3}}N \right\|_{{{L}^{2}}}^{\frac{1}{2}}\left\| {{\nabla }^{2}}N \right\|_{{{L}^{2}}}^{\frac{1}{2}}\left\| \nabla \phi  \right\|_{{{L}^{2}}}^{1/2}\left\| \phi  \right\|_{{{L}^{2}}}^{1/2} \notag\\ 
 & &+\left\| \nabla \phi  \right\|_{{{L}^{2}}}^{\frac{5}{6}}\left\| \phi  \right\|_{{{L}^{2}}}^{\frac{1}{6}}\left\| {{\nabla }^{2}}N \right\|_{{{L}^{2}}}^{3/4}\left\| \nabla N \right\|_{{{L}^{2}}}^{1/4}+\sum\limits_{m=2}^{M}{\sum\limits_{q=2}^{Q}{\left( \left\| \nabla \phi  \right\|_{{{L}^{2}}}^{\frac{8m-9}{4}}\left\| \phi  \right\|_{{{L}^{2}}}^{\frac{1}{4}} \right.}} \notag\\ 
 && \cdot \left. \left\| \nabla N \right\|_{{{L}^{2}}}^{\frac{4q-5}{4}}\left\| N \right\|_{{{L}^{2}}}^{\frac{1}{4}}\left\| {{\nabla }^{2}}\phi  \right\|_{{{L}^{2}}}^{3/4}\left\| \nabla \phi  \right\|_{{{L}^{2}}}^{1/4}\left\| {{\nabla }^{2}}N \right\|_{{{L}^{2}}}^{3/4}\left\| \nabla N \right\|_{{{L}^{2}}}^{1/4} \right) \notag\\ 
 && +\sum\limits_{m=2}^{M}{\sum\limits_{q=2}^{Q}{\left\| \nabla \phi  \right\|_{{{L}^{2}}}^{\frac{8m-13}{4}}\left\| \phi  \right\|_{{{L}^{2}}}^{\frac{1}{4}}\left\| \nabla N \right\|_{{{L}^{2}}}^{\frac{4q-1}{4}}\left\| N \right\|_{{{L}^{2}}}^{\frac{1}{4}}\left\| {{\nabla }^{2}}\phi  \right\|_{{{L}^{2}}}^{3/2}\left\| \nabla \phi  \right\|_{{{L}^{2}}}^{1/2}}} \notag\\ 
 && +\sum\limits_{m=2}^{M}{\sum\limits_{q=2}^{Q}{\left\| \nabla \phi  \right\|_{{{L}^{2}}}^{\frac{8m-9}{4}}\left\| \phi  \right\|_{{{L}^{2}}}^{\frac{1}{4}}\left\| \nabla N \right\|_{{{L}^{2}}}^{\frac{4q-1}{4}}\left\| N \right\|_{{{L}^{2}}}^{\frac{1}{4}}\left\| {{\nabla }^{2}}\phi  \right\|_{{{L}^{2}}}^{1/2}\left\| {{\nabla }^{3}}\phi  \right\|_{{{L}^{2}}}^{1/2}}} \notag\\ 
 && +\sum\limits_{m=3}^{M}{\sum\limits_{q=3}^{Q}{\left\| \nabla \phi  \right\|_{{{L}^{2}}}^{\frac{8m-5}{4}}\left\| \phi  \right\|_{{{L}^{2}}}^{\frac{1}{4}}\left\| \nabla N \right\|_{{{L}^{2}}}^{\frac{4q-9}{4}}\left\| N \right\|_{{{L}^{2}}}^{\frac{1}{4}}\left\| {{\nabla }^{2}}N \right\|_{{{L}^{2}}}^{3/4}\left\| \nabla N \right\|_{{{L}^{2}}}^{1/4}}} \notag\\ 
 && +\sum\limits_{m=3}^{M}{\sum\limits_{q=3}^{Q}{\left\| \nabla \phi  \right\|_{{{L}^{2}}}^{\frac{8m-5}{4}}\left\| \phi  \right\|_{{{L}^{2}}}^{\frac{1}{4}}\left\| \nabla N \right\|_{{{L}^{2}}}^{\frac{4q-5}{4}}\left\| N \right\|_{{{L}^{2}}}^{\frac{1}{4}}\left\| {{\nabla }^{3}}N \right\|_{{{L}^{2}}}^{1/2}\left\| {{\nabla }^{2}}N \right\|_{{{L}^{2}}}^{1/2}}} \notag\\ 
 & \le& C{{\left( 1+t \right)}^{2}}~.  
\end{eqnarray}
This concludes the proof of the lemma.
\end{proof}
The same procedure also apply for $\Tilde{k}=4$, we can show
\begin{eqnarray}
   {{\left\| {{\nabla }^{4}}A \right\|}_{{{L}^{2}}}}&\le& {{\left\| \nabla A\nabla A\phi  \right\|}_{{{L}^{2}}}}+{{\left\| A\phi {{\nabla }^{2}}A \right\|}_{{{L}^{2}}}}+{{\left\| A\nabla A\nabla \phi  \right\|}_{{{L}^{2}}}}+{{\left\| {{A}^{2}}\nabla^2 \phi  \right\|}_{{{L}^{2}}}}+{{\left\| {{\nabla }^{3}}A \right\|}_{{{L}^{2}}}} \notag\\ 
 & \le& {{\left\| \nabla A \right\|}_{{{L}^{4}}}}{{\left\| \phi  \right\|}_{{{L}^{8}}}}{{\left\| \nabla A \right\|}_{{{L}^{8}}}}+{{\left\| {{\nabla }^{2}}A \right\|}_{{{L}^{4}}}}{{\left\| \phi  \right\|}_{{{L}^{8}}}}{{\left\| A \right\|}_{{{L}^{8}}}} \notag\\ 
 && +{{\left\| \nabla A \right\|}_{{{L}^{4}}}}{{\left\| \nabla \phi  \right\|}_{{{L}^{8}}}}{{\left\| A \right\|}_{{{L}^{8}}}}+{{\left\| \nabla^2 \phi  \right\|}_{{{L}^{4}}}}{{\left\| {{A}^{2}} \right\|}_{{{L}^{4}}}}+{{\left\| {{\nabla }^{3}}A \right\|}_{{{L}^{2}}}} \notag\\ 
 & \le& \left\| {{\nabla }^{2}}A \right\|_{{{L}^{2}}}^{1/2}\left\| \nabla A \right\|_{{{L}^{2}}}^{1/2}\left\| \nabla \phi  \right\|_{{{L}^{2}}}^{3/4}\left\| \phi  \right\|_{{{L}^{2}}}^{1/4}\left\| {{\nabla }^{2}}A \right\|_{{{L}^{2}}}^{3/4}\left\| \nabla A \right\|_{{{L}^{2}}}^{1/4} \notag\\ 
 && +\left\| {{\nabla }^{3}}A \right\|_{{{L}^{2}}}^{1/2}\left\| {{\nabla }^{2}}A \right\|_{{{L}^{2}}}^{1/2}\left\| \nabla \phi  \right\|_{{{L}^{2}}}^{3/4}\left\| \phi  \right\|_{{{L}^{2}}}^{1/4}\left\| \nabla A \right\|_{{{L}^{2}}}^{3/4}\left\| A \right\|_{{{L}^{2}}}^{1/4} \notag\\ 
 && +\left\| {{\nabla }^{2}}A \right\|_{{{L}^{2}}}^{1/2}\left\| \nabla A \right\|_{{{L}^{2}}}^{1/2}\left\| {{\nabla }^{2}}\phi  \right\|_{{{L}^{2}}}^{3/4}\left\| \nabla \phi  \right\|_{{{L}^{2}}}^{1/4}\left\| \nabla A \right\|_{{{L}^{2}}}^{3/4}\left\| A \right\|_{{{L}^{2}}}^{1/4} \notag\\ 
 && +\left\| {{\nabla }^{3}}\phi  \right\|_{{{L}^{2}}}^{1/2}\left\| {{\nabla }^{2}}\phi  \right\|_{{{L}^{2}}}^{1/2}\left\| \nabla A \right\|_{{{L}^{2}}}^{3/2}\left\| A \right\|_{{{L}^{2}}}^{1/2}+{{\left\| {{\nabla }^{3}}A \right\|}_{{{L}^{2}}}} \notag\\ 
 & \le& C{{\left( 1+t \right)}^{2}}~.  
\end{eqnarray}
Using the results in lemma \ref{lemmaV}, one can also derive
\begin{eqnarray}
   {{\left\| {{\nabla }^{4}}\phi  \right\|}_{{{L}^{2}}}}&\le& {{\left\| {{\nabla }^{2}}{{\partial }_{\phi }}V \right\|}_{{{L}^{2}}}}+{{\left\| {{\nabla }^{3}}A\phi  \right\|}_{{{L}^{2}}}}+{{\left\| {{\nabla }^{2}}A\nabla \phi  \right\|}_{{{L}^{2}}}}+{{\left\| \nabla A{{\nabla }^{2}}\phi  \right\|}_{{{L}^{2}}}}+{{\left\| A{{\nabla }^{3}}\phi  \right\|}_{{{L}^{2}}}} \notag\\ 
 && +{{\left\| \nabla A\nabla A\phi  \right\|}_{{{L}^{2}}}}+{{\left\| A{{\nabla }^{2}}A\phi  \right\|}_{{{L}^{2}}}}+{{\left\| A\nabla A\nabla \phi  \right\|}_{{{L}^{2}}}}+{{\left\| {{A}^{2}}{{\nabla }^{2}}\phi  \right\|}_{{{L}^{2}}}} \notag\\ 
 & \le& {{\left\| {{\nabla }^{2}}{{\partial }_{\phi }}V \right\|}_{{{L}^{2}}}}+{{\left\| {{\nabla }^{3}}A \right\|}_{{{L}^{4}}}}{{\left\| \phi  \right\|}_{{{L}^{4}}}}+{{\left\| {{\nabla }^{2}}A \right\|}_{{{L}^{4}}}}{{\left\| \nabla \phi  \right\|}_{{{L}^{4}}}}+{{\left\| \nabla A \right\|}_{{{L}^{4}}}}{{\left\| {{\nabla }^{2}}\phi  \right\|}_{{{L}^{4}}}} \notag\\ 
 && +{{\left\| A \right\|}_{{{L}^{4}}}}{{\left\| {{\nabla }^{3}}\phi  \right\|}_{{{L}^{4}}}}+{{\left\| \nabla A \right\|}_{{{L}^{4}}}}{{\left\| \nabla A \right\|}_{{{L}^{8}}}}{{\left\| \phi  \right\|}_{{{L}^{8}}}}+{{\left\| {{\nabla }^{2}}A \right\|}_{{{L}^{4}}}}{{\left\| A \right\|}_{{{L}^{8}}}}{{\left\| \phi  \right\|}_{{{L}^{8}}}} \notag\\ 
 && +{{\left\| \nabla A \right\|}_{{{L}^{4}}}}{{\left\| A \right\|}_{{{L}^{8}}}}{{\left\| \nabla \phi  \right\|}_{{{L}^{8}}}}+{{\left\| {{A}^{2}} \right\|}_{{{L}^{4}}}}{{\left\| {{\nabla }^{2}}\phi  \right\|}_{{{L}^{4}}}} \notag\\ 
 & \le& {{\left\| {{\nabla }^{2}}{{\partial }_{\phi }}V \right\|}_{{{L}^{2}}}}+\left\| {{\nabla }^{4}}A \right\|_{{{L}^{2}}}^{1/2}\left\| {{\nabla }^{3}}A \right\|_{{{L}^{2}}}^{1/2}\left\| \nabla \phi  \right\|_{{{L}^{2}}}^{1/2}\left\| \phi  \right\|_{{{L}^{2}}}^{1/2} \notag\\ 
 && +\left\| {{\nabla }^{3}}A \right\|_{{{L}^{2}}}^{1/2}\left\| {{\nabla }^{2}}A \right\|_{{{L}^{2}}}^{1/2}\left\| {{\nabla }^{2}}\phi  \right\|_{{{L}^{2}}}^{1/2}\left\| \nabla \phi  \right\|_{{{L}^{2}}}^{1/2} \notag\\ 
 && +\left\| {{\nabla }^{2}}A \right\|_{{{L}^{2}}}^{1/2}\left\| \nabla A \right\|_{{{L}^{2}}}^{1/2}\left\| {{\nabla }^{3}}\phi  \right\|_{{{L}^{2}}}^{1/2}\left\| {{\nabla }^{2}}\phi  \right\|_{{{L}^{2}}}^{1/2} \notag\\ 
 && +\left\| \nabla A \right\|_{{{L}^{2}}}^{1/2}\left\| A \right\|_{{{L}^{2}}}^{1/2}\left\| {{\nabla }^{4}}\phi  \right\|_{{{L}^{2}}}^{1/2}\left\| {{\nabla }^{3}}\phi  \right\|_{{{L}^{2}}}^{1/2} \notag\\ 
 && +\left\| {{\nabla }^{2}}A \right\|_{{{L}^{2}}}^{1/2}\left\| \nabla A \right\|_{{{L}^{2}}}^{1/2}\left\| {{\nabla }^{2}}A \right\|_{{{L}^{2}}}^{3/4}\left\| \nabla A \right\|_{{{L}^{2}}}^{1/4}\left\| \nabla \phi  \right\|_{{{L}^{2}}}^{3/4}\left\| \phi  \right\|_{{{L}^{2}}}^{1/4} \notag\\ 
 && +\left\| {{\nabla }^{3}}A \right\|_{{{L}^{2}}}^{1/2}\left\| {{\nabla }^{2}}A \right\|_{{{L}^{2}}}^{1/2}\left\| \nabla \phi  \right\|_{{{L}^{2}}}^{3/4}\left\| \phi  \right\|_{{{L}^{2}}}^{1/4}\left\| \nabla A \right\|_{{{L}^{2}}}^{3/4}\left\| A \right\|_{{{L}^{2}}}^{1/4} \notag\\ 
 && +\left\| {{\nabla }^{2}}A \right\|_{{{L}^{2}}}^{1/2}\left\| \nabla A \right\|_{{{L}^{2}}}^{1/2}\left\| \nabla A \right\|_{{{L}^{2}}}^{3/4}\left\| A \right\|_{{{L}^{2}}}^{1/4}\left\| {{\nabla }^{2}}\phi  \right\|_{{{L}^{2}}}^{3/4}\left\| \nabla \phi  \right\|_{{{L}^{2}}}^{1/4} \notag\\ 
 && +\left\| \nabla A \right\|_{{{L}^{2}}}^{3/2}\left\| A \right\|_{{{L}^{2}}}^{1/2}\left\| {{\nabla }^{3}}\phi  \right\|_{{{L}^{2}}}^{1/2}\left\| {{\nabla }^{2}}\phi  \right\|_{{{L}^{2}}}^{1/2} \notag\\
 &\le& C(1+t)^2~.
\end{eqnarray}
Hence, the proposition is proved.
\end{proof}
Next, we consider the initial conditions at $t=0$
\begin{eqnarray}\label{initialdata}
    A\left( 0,x \right)&=&{{a}_{\left( 0 \right)}}\left( x \right),\hspace{2em}{{\partial }_{t}}A\left( 0,x \right)={{a}_{\left( 1 \right)}}\left( x \right)~,\notag\\
    \phi \left( 0,x \right)&=&{{\varphi }_{\left( 0 \right)}}\left( x \right),\hspace{2em}{{\partial }_{t}}\phi \left( 0,x \right)={{\varphi }_{\left( 1 \right)}}\left( x \right)~,\notag\\
    N \left( 0,x \right)&=&{{n }_{\left( 0 \right)}}\left( x \right),\hspace{2em}{{\partial }_{t}}N \left( 0,x \right)={{n }_{\left( 1 \right)}}\left( x \right)~,\notag\\
    \text{div}{{a}_{\left( 0 \right)}}&=&0,\hspace{5.8em}\text{div}{{a}_{\left( 1 \right)}}=0.
\end{eqnarray}
In this case, no initial conditions are required for $A_0 $. The dynamic variables are simply $A, \phi$ and $N$. The norm that governs the system's evolution in equations \eqref{eomCG1}-\eqref{eomCG3} depends solely on these variables. The main results of our work are presented as follows.
\begin{namedtheorem}[Main]
For general initial data $a_{(0)}, a_{(1)}, \varphi_{(0)}, \varphi_{(1)}, n_{(0)}, n_{(1)}$ in \eqref{initialdata}, let
\begin{eqnarray}
      {{\mathcal{J}}_{0}}&=&{{\left\| \nabla {{a}_{\left( 0 \right)}} \right\|}_{{{L}^{2}}\left( {{\mathbb{R}}^{2}} \right)}}+{{\left\| {{a}_{\left( 1 \right)}} \right\|}_{{{L}^{2}}\left( {{\mathbb{R}}^{2}} \right)}}+{{\left\| {{\varphi }_{\left( 0 \right)}} \right\|}_{{{L}^{2}}\left( {{\mathbb{R}}^{2}} \right)}}+{{\left\| \nabla {{\varphi }_{\left( 0 \right)}} \right\|}_{{{L}^{2}}\left( {{\mathbb{R}}^{2}} \right)}} \notag\\ 
 && +{{\left\| {{\varphi }_{\left( 1 \right)}} \right\|}_{{{L}^{2}}\left( {{\mathbb{R}}^{2}} \right)}}+{{\left\| \nabla {{n}_{\left( 0 \right)}} \right\|}_{{{L}^{2}}\left( {{\mathbb{R}}^{2}} \right)}}+{{\left\| {{n}_{\left( 1 \right)}} \right\|}_{{{L}^{2}}\left( {{\mathbb{R}}^{2}} \right)}} 
\end{eqnarray}
is finite. Under this condition, there exists a unique global generalized solution to the MCSH equations \eqref{eom3}, \eqref{eomCG1},\eqref{eomCG3}, and  \eqref{eomCG2P} verifying the energy inequality
\begin{eqnarray}
\mathcal{E}\left(A_0,A,\phi\right)\le\mathcal{J}_0
\end{eqnarray}
as well as
\begin{eqnarray}\label{i}
\mathcal{J}(t)\le \Tilde{C}\left(1+t\right)^2
\end{eqnarray}
where $\Tilde{C}$ dependent only on $ \mathcal{J}_0 $ and on any finite interval $[0,T]$
\begin{eqnarray}\label{ii}
    \int\limits_{0}^{T}{\left( {{\left\| \square A \right\|}_{{{L}^{2}}}}+{{\left\| \square \phi  \right\|}_{{{L}^{2}}}}+{{\left\| \square N \right\|}_{{{L}^{2}}}} \right)}\text{ }dt<\infty ~.
\end{eqnarray}
Furthermore, if the initial data satisfy the higher regularity assumptions $\nabla {{a}_{\left( 0 \right)}}\in {{H}^{s}}\left( {{\mathbb{R}}^{2}} \right)$, ${{a}_{\left( 1 \right)}}\in {{H}^{s}}\left( {{\mathbb{R}}^{2}} \right),{{\varphi }_{\left( 0 \right)}}\in {{H}^{s+1}}\left( {{\mathbb{R}}^{2}} \right),{{\varphi }_{\left( 1 \right)}}\in {{H}^{s}}\left( {{\mathbb{R}}^{2}} \right)$ ,$\nabla {{n}_{\left( 0 \right)}}\in {{H}^{s}}\left( {{\mathbb{R}}^{2}} \right)$ and ${{n}_{\left( 1 \right)}}\in {{H}^{s}}\left( {{\mathbb{R}}^{2}} \right)$ for some integer $s>0$ then also for any time $t>0$, ${{A}_{0}},A\in {{H}^{s+1}}\left( {{\mathbb{R}}^{2}} \right),{{\partial }_{t}}{{A}_{0}},{{\partial }_{t}}A\in {{H}^{s}}\left( {{\mathbb{R}}^{2}} \right),\phi,N \in {{H}^{s+1}}\left( {{\mathbb{R}}^{2}} \right),{{\partial }_{t}}\phi,\partial_t N \in {{H}^{s}}\left( {{\mathbb{R}}^{2}} \right)$.
\end{namedtheorem}
\section{The Energy Estimate}\label{sec:MCSH estimate}
\subsection{Estimate for \texorpdfstring{$A_0$}{}}
To prove the main theorem, it is essential to first obtain estimates for the fields. This subsection focuses on deriving estimates for the static component of the gauge potential $A_0$. 
\begin{preposition}\label{prepA0}
    Let $A_0$ be the solution of equation\eqref{eomCG1}, for $A$ and $\phi$ satisfies condition \eqref{J}, then at each time $t\in[0,T]$, we have
    \begin{eqnarray}
        {{\left\| \nabla {{A}_{0}} \right\|}_{{{L}^{2}}}}+{{\left\| {{A}_{0}}\phi  \right\|}_{{{L}^{2}}}}\le C(1+t)^2\mathcal{J}\left( t \right)~,
    \end{eqnarray}
    \begin{eqnarray}
        {{\left\| \nabla {{A}_{0}} \right\|}_{{{L}^{3}}}}\le C(1+t)\mathcal{J}(t)^{1/3}~,
    \end{eqnarray}
    \begin{eqnarray}
        {{\left\| {{A}_{0}} \right\|}_{{{L}^{\infty }}}}\le C{{\left( 1+t \right)}^{2}}\mathcal{J}{{\left( t \right)}^{2}}~.
    \end{eqnarray}
\end{preposition}
\begin{proof}
To prove preposition \ref{prepA0},  we can write equation \eqref{eomCG1} in the form
\begin{eqnarray}\label{deltaA0}
    \Delta {{A}^{0}}-{{A}^{0}}{{\left| \phi  \right|}^{2}}=\kappa {{F}^{0}}-\operatorname{Im}\left( \phi {{\partial }_{t}}\bar{\phi } \right)~.
\end{eqnarray}   
Multiplying  with $A_0$ and then integrating by parts, we obtain the following result
\begin{eqnarray}\label{A0esti}
 {{\int\limits_{{{\mathbb{R}}^{2}}}{\left| \nabla {{A}_{0}} \right|^{2}}}}d{{x}^{2}}+\int\limits_{{{\mathbb{R}}^{2}}}{{{\left| {{A}_{0}}\phi  \right|}^{2}}}d{{x}^{2}}\le c\left( \int\limits_{{{\mathbb{R}}^{2}}}{\left| {{F}_{0}}{{A}_{0}} \right|}d{{x}^{2}}+\int\limits_{{{\mathbb{R}}^{2}}}{\left| {{A}_{0}}\phi {{\partial }_{t}}\phi  \right|}d{{x}^{2}} \right) ~.
\end{eqnarray}
Applying Cauchy-Schwarz inequality on each term of \eqref{A0esti} along with Gagliardo-Nirenberg-Sobolev inequality, we attain
\begin{eqnarray}
 \int\limits_{{{\mathbb{R}}^{2}}}{{{\left| \nabla {{A}_{0}} \right|}^{2}}}d{{x}^{2}}+\int\limits_{{{\mathbb{R}}^{2}}}{{{\left| {{A}_{0}}\phi  \right|}^{2}}}d{{x}^{2}} & \le &  {{\left\| {{A}_{0}} \right\|}_{{{L}^{2}}}}{{\left\| \nabla A \right\|}_{{{L}^{2}}}}+{{\left\| {{A}_{0}}\phi  \right\|}_{{{L}^{2}}}}{{\left\| {{\partial }_{t}}\phi  \right\|}_{{{L}^{2}}}} \notag\\ 
 & \le &  {{\left\| {{A}_{0}} \right\|}_{{{L}^{2}}}}{{\left\| \nabla A \right\|}_{{{L}^{2}}}}+{{\left\| {{A}_{0}} \right\|}_{{{L}^{4}}}}{{\left\| \phi  \right\|}_{{{L}^{4}}}}{{\left\| {{\partial }_{t}}\phi  \right\|}_{{{L}^{2}}}} \notag\\ 
 & \le &  {{\left\| {{A}_{0}} \right\|}_{{{L}^{2}}}}{{\left\| \nabla A \right\|}_{{{L}^{2}}}}+\left\| \nabla {{A}_{0}} \right\|_{{{L}^{2}}}^{1/2}\left\| {{A}_{0}} \right\|_{{{L}^{2}}}^{1/2}\left\| \nabla \phi  \right\|_{{{L}^{2}}}^{1/2}\left\| \phi  \right\|_{{{L}^{2}}}^{1/2}{{\left\| {{\partial }_{t}}\phi  \right\|}_{{{L}^{2}}}}~. \notag\\ 
\end{eqnarray}
Using elliptic regularity theorem for Poisson equation, we have
\begin{eqnarray}
  {{\left\| {{\partial }_{t}}{{A}_{0}} \right\|}_{{{L}^{2}}}}& \le&  {{\left\| {{\partial }_{t}}{{F}_{ij}} \right\|}_{{{L}^{2}}}}+{{\left\| {{\partial }_{\phi }}V \right\|}_{{{L}^{2}}}} \notag\\ 
 & \le&  {{\left\| {{\partial }_{t}}{{F}_{ij}} \right\|}_{{{L}^{2}}}}+{{\left\| \sum\limits_{m=1}^{M}{\alpha {{\left| \phi  \right|}^{2m}}}\sum\limits_{q=1}^{Q}{\beta {{N}^{q}}} \right\|}_{{{L}^{2}}}} \notag\\ 
 & \le&  {{\left\| {{\partial }_{t}}{{F}_{ij}} \right\|}_{{{L}^{2}}}}+\sum\limits_{m=1}^{M}{\left\| \phi  \right\|_{{{L}^{8m}}}^{2m}}\sum\limits_{m=1}^{M}{\left\| N \right\|_{{{L}^{4q}}}^{q}} \notag\\ 
 & \le&  {{E}_{0}}+\sum\limits_{m=1}^{M}{\left\| \nabla \phi  \right\|_{{{L}^{2}}}^{\frac{4m-1}{2}}}\sum\limits_{q=1}^{Q}{\left\| \nabla N \right\|_{{{L}^{2}}}^{\frac{2q-1}{2}}}\left\| N \right\|_{{{L}^{2}}}^{\frac{1}{2}}\left\| \phi  \right\|_{{{L}^{2}}}^{1/2} \notag\\ 
 & \le&  C\left( 1+t \right)~.  
\end{eqnarray}
We can derive the estimate for ${{\left\| {{A}_{0}} \right\|}_{{{L}^{2}}}}$ as follows
\begin{eqnarray}
   \frac{d}{dt}\left\| {{A}_{0}} \right\|_{{{L}^{2}}}^{2}&=&\int\limits_{{{R}^{2}}}{2{{A}_{0}}{{\partial }_{t}}{{A}_{0}}}dx \notag\\ 
 & \le& C{{\left\| {{A}_{0}} \right\|}_{{{L}^{2}}}}{{\left\| {{\partial }_{t}}{{A}_{0}} \right\|}_{{{L}^{2}}}} \notag\\ 
  \frac{d}{dt}{{\left\| {{A}_{0}} \right\|}_{{{L}^{2}}}}&\le& C{{\left\| {{\partial }_{t}}{{A}_{0}} \right\|}_{{{L}^{2}}}} \notag\\ 
 & \le& C\left( 1+t \right) \notag\\ 
 {{\left\| {{A}_{0}} \right\|}_{{{L}^{2}}}}&\le& C{{\left( 1+t \right)}^{2}} ~. 
\end{eqnarray}
Therefore, we arrive at
\begin{eqnarray}
    {{\left\| \nabla {{A}_{0}} \right\|}_{{{L}^{2}}}}+{{\left\| {{A}_{0}}\phi  \right\|}_{{{L}^{2}}}}\le {{\left( 1+t \right)}^{2}}\left( {{\left\| \partial A \right\|}_{{{L}^{2}}}}+{{\left\| {{\partial }_{t}}\phi  \right\|}_{{{L}^{2}}}} \right)\le C{{\left( 1+t \right)}^{2}} \mathcal{J}(t)~.
\end{eqnarray}
Next, applying Sobolev-type inequality to Poisson equation to \eqref{eomCG1}, we derive
\begin{eqnarray}
   {{\left\| \nabla {{A}_{0}} \right\|}_{{{L}^{3}}}}&\le& C\left( {{\left\| {{F}_{ij}} \right\|}_{{{L}^{6/5}}}}+{{\left\| \phi \overline{{{D}_{0}}\phi } \right\|}_{{{L}^{6/5}}}} \right) \notag\\ 
 & \le &C\left( {{\left\| \nabla A \right\|}_{{{L}^{6/5}}}}+{{\left\| \phi  \right\|}_{{{L}^{3}}}}{{\left\| \overline{{{D}_{0}}\phi } \right\|}_{{{L}^{2}}}} \right) \notag\\ 
 & \le &C\left( \left\| \nabla A \right\|_{{{L}^{2}}}^{1/3}\left\| A \right\|_{{{L}^{2}}}^{2/3}+\left\| \nabla \phi  \right\|_{{{L}^{2}}}^{1/3}\left\| \phi  \right\|_{{{L}^{2}}}^{2/3}{{\left\| \overline{{{D}_{0}}\phi } \right\|}_{{{L}^{2}}}} \right) \notag\\ 
 & \le &C\left( 1+t \right)\left( \left\| \nabla A \right\|_{{{L}^{2}}}^{1/3}+\left\| \nabla \phi  \right\|_{{{L}^{2}}}^{1/3} \right)~.  
\end{eqnarray}
The estimate for ${{\left\| {{A}_{0}} \right\|}_{{{L}^{\infty }}}}$ is obtained in a similar procedure
\begin{eqnarray}\label{A0infnty}
  {{\left\| {{A}_{0}} \right\|}_{{{L}^{\infty }}}}&\le& C\left( {{\left\| \partial A \right\|}_{{{L}^{2}}}}+{{\left\| \phi D\phi  \right\|}_{{{L}^{2}}}}+{{\left\| {{A}_{0}} \right\|}_{{{L}^{2}}}} \right) \notag\\ 
 & \le& C\left( {{\left\| \partial A \right\|}_{{{L}^{2}}}}+{{\left\| \phi {{\partial }_{t}}\phi  \right\|}_{{{L}^{2}}}}+{{\left\| {{A}_{0}}{{\phi }^{2}} \right\|}_{{{L}^{2}}}}+{{\left\| {{A}_{0}} \right\|}_{{{L}^{2}}}} \right) \notag\\ 
 & \le& C\left( {{\left\| \partial A \right\|}_{{{L}^{2}}}}+{{\left\| \phi  \right\|}_{{{L}^{\infty}}}}{{\left\| {{\partial }_{t}}\phi  \right\|}_{{{L}^{2}}}}+{{\left\| {{A}_{0}} \right\|}_{{{L}^{3}}}}{{\left\| {{\phi }^{2}} \right\|}_{{{L}^{6}}}}+{{\left\| {{A}_{0}} \right\|}_{{{L}^{2}}}} \right) \notag\\ 
 & \le& C\left( {{\left\| \partial A \right\|}_{{{L}^{2}}}}+{{\left\| {{\partial }_{t}}\phi  \right\|}_{{{L}^{2}}}}+\left\| \nabla {{A}_{0}} \right\|_{{{L}^{2}}}^{1/3}\left\| {{A}_{0}} \right\|_{{{L}^{2}}}^{2/3}\left\| \phi  \right\|_{{{L}^{4}}}^{2/3}+{{\left\| {{A}_{0}} \right\|}_{{{L}^{2}}}} \right) \notag\\ 
 & \le& C{{\left( 1+t \right)}^{2}}\mathcal{J}{{\left( t \right)}^{2}} ~.
\end{eqnarray}
\end{proof}
Next, we derive estimates for two distinct solutions of  \eqref{eomCG1}. Specifically, we consider two distinct solutions of $A_0, A, \phi$ and $A_0', A', \phi' $ that verify \eqref{eomCG1} and \eqref{eomCG3}.
\begin{preposition}\label{prepo1}
Consider two solutions $A_0, A, \phi$ and $A_0', A', \phi' $ to \eqref{eomCG1} on a slab $[0, T^*] \times \mathbb{R}^2$ corresponding to the same initial data $a_0 \in H^1(\mathbb{R}^2), a_1 \in L^2(\mathbb{R}^2), \phi_0 \in H^1(\mathbb{R}^2), \phi_1 \in L^2(\mathbb{R}^2)$. Thus, $\mathcal{J}_0$ is finite. Then $A_0 = A_0'$ in $[0, T^*]\times \mathbb{R}^2$.
\end{preposition}
\begin{proof}
 By subtracting the equations \eqref{eomCG1} governing $A_0$ and $A_0'$, we derive:
 \begin{eqnarray}\label{A0-A0}
     \Delta \left( {{A}_{0}}-{{{{A}'}}_{0}} \right)-\left( {{A}_{0}}{{\left| \phi  \right|}^{2}}-{{{{A}'}}_{0}}{{\left| {{\phi }'} \right|}^{2}} \right)=Im~\left( {\phi }'\text{ }\overline{{{\partial }_{t}}{\phi }'}-\phi \text{ }\overline{{{\partial }_{t}}\phi } \right)+\kappa \left( {{F}_{0}}-{{{{F}'}}_{0}} \right)~.
 \end{eqnarray}
 By multiplying \eqref{A0-A0} by $A_0-A_0'$ and integrating over $\mathbb{R}^2$, following standard integration by parts, we obtain  
\begin{eqnarray}
   \int_{{{\mathbb{R}}^{2}}}{\left( {{\left| \nabla \left( {{A}_{0}}-{{{{A}'}}_{0}} \right) \right|}^{2}}+{{\left| {{A}_{0}}\phi -{{{{A}'}}_{0}}{\phi }' \right|}^{2}} \right)}dx&\le& \int_{{{\mathbb{R}}^{2}}}{\left| \left( {{A}_{0}}\phi -{{{{A}'}}_{0}}{\phi }' \right)\left( {{A}_{0}}-{{{{A}'}}_{0}} \right)\left( \phi -{\phi }' \right) \right|}dx \notag\\ 
 && +\int_{{{\mathbb{R}}^{2}}}{\left| \left( {{A}_{0}}-{{{{A}'}}_{0}} \right){{{{A}'}}_{0}}{\phi }'\left( \phi -{\phi }' \right) \right|}dx \notag\\ 
 && +\int_{{{\mathbb{R}}^{2}}}{\left| \left( {{A}_{0}}-{{{{A}'}}_{0}} \right)\left( \phi -{\phi }' \right){{\partial }_{t}}\phi  \right|}dx \notag\\ 
 && +\int_{{{\mathbb{R}}^{2}}}{\left| \left( {{A}_{0}}-{{{{A}'}}_{0}} \right){\phi }'\left( {{\partial }_{t}}\phi -{{\partial }_{t}}{\phi }' \right){{\partial }_{t}}\phi  \right|}dx \notag\\ 
 && +\int_{{{\mathbb{R}}^{2}}}{\left| \left( {{A}_{0}}-{{{{A}'}}_{0}} \right)\left( {{F}_{0}}-{{{{F}'}}_{0}} \right) \right|}dx \notag\\
  & \le& C\left( {{\left\| {{A}_{0}}\phi -{{{{A}'}}_{0}}{\phi }' \right\|}_{{{L}^{2}}}}{{\left\| {{A}_{0}}-{{{{A}'}}_{0}} \right\|}_{{{L}^{3}}}}{{\left\| \phi -{\phi }' \right\|}_{{{L}^{6}}}} \right. \notag\\ 
 && +{{\left\| {{{{A}'}}_{0}} \right\|}_{{{L}^{6}}}}{{\left\| {{\phi }'} \right\|}_{{{L}^{2}}}}{{\left\| {{A}_{0}}-{{{{A}'}}_{0}} \right\|}_{{{L}^{6}}}}{{\left\| \phi -{\phi }' \right\|}_{{{L}^{6}}}} \notag\\ 
 && +{{\left\| {{A}_{0}}-{{{{A}'}}_{0}} \right\|}_{{{L}^{6}}}}{{\left\| \phi -{\phi }' \right\|}_{{{L}^{3}}}}{{\left\| {{\partial }_{t}}\phi  \right\|}_{{{L}^{2}}}} \notag\\ 
 && +{{\left\| \phi -{\phi }' \right\|}_{{{L}^{3}}}}{{\left\| {{A}_{0}}-{{{{A}'}}_{0}} \right\|}_{{{L}^{6}}}}{{\left\| {{\partial }_{t}}\phi -{{\partial }_{t}}{\phi }' \right\|}_{{{L}^{2}}}} \notag\\ 
 && \left. +{{\left\| {{A}_{0}}-{{{{A}'}}_{0}} \right\|}_{{{L}^{3}}}}{{\left\| {{F}_{0}}-{{{{F}'}}_{0}} \right\|}_{{{L}^{3/2}}}} \right) \notag\\
  & \le& C\left( {{\left\| \nabla \left( {{A}_{0}}-{{{{A}'}}_{0}} \right) \right\|}_{{{L}^{2}}}}+{{\left\| {{A}_{0}}\phi -{{{{A}'}}_{0}}{\phi }' \right\|}_{{{L}^{2}}}} \right) \notag\\ 
 && {{\left( 1+t \right)}^{2}}\left( \mathcal{J}+{\mathcal{J}}' \right)\mathcal{J}\left( A-{A}',\phi -{\phi }' \right)\left( t \right)~.  \notag\\
\end{eqnarray}
\normalsize
Similarly, applying Sobolev-type inequality to Poisson equation satisfied by $A_0-A_0'$ derived from \eqref{deltaA0}, we have
\begin{eqnarray}
  {{\left\| \nabla \left( {{A}_{0}}-{{{{A}'}}_{0}} \right) \right\|}_{{{L}^{3}}}} &\le & C\left( {{\left\| {{A}_{0}}\phi -{{{{A}'}}_{0}}{\phi }' \right\|}_{{{L}^{2}}}}{{\left\| \phi  \right\|}_{{{L}^{3}}}}+{{\left\| {{{{A}'}}_{0}}{\phi }' \right\|}_{{{L}^{2}}}}{{\left\| \phi -{\phi }' \right\|}_{{{L}^{3}}}} \right.+{{\left\| \phi -{\phi }' \right\|}_{{{L}^{3}}}}{{\left\| {{\partial }_{t}}\phi  \right\|}_{{{L}^{2}}}} \notag\\ 
 & & \left. +{{\left\| {{\phi }'} \right\|}_{{{L}^{3}}}}{{\left\| {{\partial }_{t}}\left( \phi -{\phi }' \right) \right\|}_{{{L}^{2}}}}+{{\left\| {{A}_{0}}-{{{{A}'}}_{0}} \right\|}_{{{L}^{2}}}}{{\left\| {{F}_{0}}-{{{{F}'}}_{0}} \right\|}_{{{L}^{3}}}} \right) \notag\\ 
 & \le & C\left( 1+t \right)\mathcal{J}{{\left( A-{A}',\phi -{\phi }' \right)}^{1/3}}\left( t \right) ~. 
\end{eqnarray}
Next, using the same procedure as derivation of \eqref{A0infnty}, ones can show that
\begin{eqnarray}
   {{\left\|  \left( {{A}_{0}}-{{{{A}'}}_{0}} \right) \right\|}_{{{L}^{\infty}}}} \le C\left( 1+t \right)^2\mathcal{J}{{\left( A-{A}',\phi -{\phi }' \right)}^{2}}\left( t \right)~.  
\end{eqnarray}
Since $\mathcal{J}{{\left( A-{A}',\phi -{\phi }' \right)}}\left( t \right)=\mathcal{J}{{\left( A-{A}',\phi -{\phi }' \right)}}\left( 0 \right)$, $A_0, A, \phi$ as well as  ${A_0}', A', \phi'$ share the same initial data, it follows that  $\mathcal{J}{{\left( A-{A}',\phi -{\phi }' \right)}}\left( 0 \right)=0$. Consequently, we conclude that ${{A}_{0}}={{{{A}'}}_{0}}$. This verifies the proposition \ref{prepo1}.
\end{proof}

\subsection{Estimate for \texorpdfstring{$A, N$ and $\phi$}{}}
\begin{preposition}\label{prepoANphi}
    Let $A_0, A, \phi, N$ be a solution of the system \eqref{eom3}, \eqref{eomCG3} and \eqref{eomCG2P} in $[0,T^*]\times\mathbb{R}^2$ having $\mathcal{J}_0$ which is finite and satisfies \eqref{i} and \eqref{ii} of the main theorem. Then,
    \begin{enumerate}[label=(\alph*)]
        \item Depending only on $\mathcal{J}_0$, there exists $T\in[0,T^*]$ such that the following inequality holds
        \begin{eqnarray}
            X\left( T \right)=\int_{0}^{T}{\left( {{\left\| \square A \right\|}_{{{L}^{2}}}}+{{\left\| \square \phi  \right\|}_{{{L}^{2}}}}+{{\left\| \square N  \right\|}_{{{L}^{2}}}} \right)dt\le 1}~.
        \end{eqnarray}
    \item There exists a real constant $C'$ depending only on $\mathcal{J}_0$ and $T^*$ such that $X(T^*) \le C'$.
    \end{enumerate}
\end{preposition}
\begin{proof}
Part (b) follows directly by repeatedly applying the result from part (a) a finite number of times and observing that, for all $0\le t\le T^*$, the condition \eqref{i} of the main theorem ensures that $\mathcal{J}(t)$ remains uniformly bounded. The first claim directly follows from the estimate, 
\begin{eqnarray}\label{X}
   X\left( T \right)\le CT{{\left( 1+{{\mathcal{J}}_{0}}+X\left( T \right) \right)}^{3}}~,
\end{eqnarray}
for all $0\le T\le T^*$ and $T\leq 1$. The function $ X $ is continuous with $ X\left( 0 \right)=0 $. Choose $ T_0 $ sufficiently small such that $ C T_0(1 + \mathcal{J}_0)^3 \leq 1 $. If $ 0 \leq T \leq T_0 $ satisfies \eqref{X}, then $ X(T) \leq C T(1 + \mathcal{J}_0)^3 \leq 1 $. To prove \eqref{X}, we begin by applying the standard energy estimate to $A, N$ and $\phi$. Specifically, for every $0\le T\le T^*$, the following holds
\begin{eqnarray}
    {{\left\| \partial A\left( T,\cdot  \right) \right\|}_{{{L}^{2}}}}\le {{\left\| \partial A\left( T,0 \right) \right\|}_{{{L}^{2}}}}+\int_{0}^{T}{{{\left\| \square A\left( t,\cdot  \right) \right\|}_{{{L}^{2}}}}dt}~,
\end{eqnarray}
\begin{eqnarray}
    {{\left\| \partial \phi \left( T,\cdot  \right) \right\|}_{{{L}^{2}}}}\le {{\left\| \partial \phi \left( T,0 \right) \right\|}_{{{L}^{2}}}}+\int_{0}^{T}{{{\left\| \square \phi \left( t,\cdot  \right) \right\|}_{{{L}^{2}}}}dt}~,
\end{eqnarray}
\begin{eqnarray}
    {{\left\| \partial N \left( T,\cdot  \right) \right\|}_{{{L}^{2}}}}\le {{\left\| \partial N \left( T,0 \right) \right\|}_{{{L}^{2}}}}+\int_{0}^{T}{{{\left\| \square N \left( t,\cdot  \right) \right\|}_{{{L}^{2}}}}dt}~.
\end{eqnarray}
Therefore,
\begin{eqnarray}
    \mathcal{J}\left( A,\phi  \right)\left( T \right)\le C\left\{ \mathcal{J}\left( A,\phi  \right)\left( 0 \right)+X\left( T \right) \right\}~.
\end{eqnarray}
In view of the wave equations \eqref{eomCG3}, \eqref{eomCG2P} and \eqref{eomCG3}, we have
\begin{eqnarray}\label{Xineq}
   X\left( T \right)&\le& \int_{0}^{T}{{{\left\| Im\mathcal{P}\left( \phi \overline{D\phi } \right) \right\|}_{{{L}^{2}}}}dt}+\int_{0}^{T}{{{\left\| \phi \left( {{\partial }_{t}}{{A}_{0}} \right) \right\|}_{{{L}^{2}}}}dt} \notag\\ 
 && +\int_{0}^{T}{{{\left\| {{A}_{0}}{{\partial }_{t}}\phi  \right\|}_{{{L}^{2}}}}dt}+\int_{0}^{T}{{{\left\| A\nabla \phi  \right\|}_{{{L}^{2}}}}dt} \notag\\ 
 && +\int_{0}^{T}{{{\left\| A_{0}^{2}\phi  \right\|}_{{{L}^{2}}}}dt}+\int_{0}^{T}{{{\left\| {{A}^{2}}\phi  \right\|}_{{{L}^{2}}}}dt}+\int_{0}^{T}{{{\left\| \mathcal{P}{{F}_{i}} \right\|}_{{{L}^{2}}}}dt} \notag\\ 
 && +\int_{0}^{T}{{{\left\| {{\partial }_{\phi }}V \right\|}_{{{L}^{2}}}}dt}+\int_{0}^{T}{{{\left\| {{\partial }_{N }}V \right\|}_{{{L}^{2}}}}dt}~.
\end{eqnarray}
Furthermore, the first term of \eqref{Xineq} can be bounded as follows
\begin{eqnarray}\label{suku1}
    \int_{0}^{T}{{{\left\| Im\mathcal{P}\left( \phi \text{ }\overline{D\phi } \right) \right\|}_{{{L}^{2}}}}dt}\le C\int_{0}^{T}{\left( {{\left\| Im\mathcal{P}\left( \phi \text{ }\overline{\nabla \phi } \right) \right\|}_{{{L}^{2}}}}+{{\left\|\mathcal{P}( A{{\phi }^{2}}) \right\|}_{{{L}^{2}}}} \right)dt}~.
\end{eqnarray}
If we now make use of the corollary to Preposition 2.1 in \cite{Klainerman}, we have
\begin{eqnarray}
   \int_{0}^{T}{\left( {{\left\| Im\mathcal{P}\left( \phi \text{ }\overline{\nabla \phi } \right) \right\|}_{{{L}^{2}}}} \right)dt}&\le& C{{T}^{1/2}}{{\left( {{\mathcal{J}}_{0}}+\int_{0}^{T}{{{\left\| \square \phi \left( t,\cdot  \right) \right\|}_{{{L}^{2}}}}dt} \right)}^{2}} \notag\\ 
 & \le& C{{T}^{1/2}}{{\left( {{\mathcal{J}}_{0}}+X\left( T \right) \right)}^{2}}~.  
\end{eqnarray}
The second term in \eqref{suku1} can similarly be estimated as follows
\begin{eqnarray}
   {{\left\| \mathcal{P}A{{\phi }^{2}} \right\|}_{{{L}^{2}}}}&\le& {{\left\| A \right\|}_{{{L}^{6}}}}\left\| \phi  \right\|_{{{L}^{6}}}^{2} \notag\\ 
 & \le& C\mathcal{J} \notag\\ 
 & \le& C{{\left( {{\mathcal{J}}_{0}}+X\left( T \right) \right)}^{2}}~.  
\end{eqnarray}
Thus, the estimate for the first term in \eqref{Xineq} is given by
\begin{eqnarray}
    \int_{0}^{T}{{{\left\| Im\mathcal{P}\left( \phi \text{ }\overline{D\phi } \right) \right\|}_{{{L}^{2}}}}dt}\le C{{T}^{1/2}}{{\left( 1+{{\mathcal{J}}_{0}}+X\left( T \right) \right)}^{2}}~.
\end{eqnarray}
For the second term, Using the H\"older inequality, one can obtain
\begin{eqnarray}
  \int_{0}^{T}{{{\left\| {{\partial }_{t}}{{A}_{0}}\phi  \right\|}_{{{L}^{2}}}}dt}&\le& \int_{0}^{T}{{{\left\| {{\partial }_{t}}{{A}_{0}} \right\|}_{{{L}^{4}}}}{{\left\| \phi  \right\|}_{{{L}^{4}}}}dt} \notag\\ 
 & \le& C\mathcal{J}\left( T \right)\int_{0}^{T}{{{\left\| {{\partial }_{t}}{{A}_{0}} \right\|}_{{{L}^{4}}}}dt}~.  
\end{eqnarray}
To find the estimate for $ {{\left\| {{\partial }_{t}}{{A}_{0}} \right\|}_{{{L}^{4}}}}$, we use the Sobolev-type inequality for Poisson equation as in the previous section and get the following inequality
\begin{eqnarray}\label{doteA0L4}
    {{\left\| {{\partial }_{t}}{{A}_{0}} \right\|}_{{{L}^{4}}}}\le C\left( {{\left\| {{\partial }_{\phi }}V \right\|}_{{{L}^{4/3}}}}+{{\left\| {{\partial }^{0}}{{F}_{0}} \right\|}_{{{L}^{4/3}}}}+{{\left\| \phi {{D}_{i}}\phi  \right\|}_{{{L}^{4}}}} \right)
\end{eqnarray}
For the first term of \eqref{doteA0L4}, using the potential form in assumption \ref{asumption}, one can derive 
\begin{eqnarray}
   {{\left\| {{\partial }_{\phi }}V \right\|}_{{{L}^{4/3}}}}&\le&c {{\left\| \sum\limits_{m=1}^{M}{\sum\limits_{q=1}^{Q}{{{\left| \phi  \right|}^{2m-1}}{{N}^{q}}}} \right\|}_{{{L}^{4/3}}}} 
   \notag\\ 
 & \le&c \sum\limits_{m=1}^{M}{\sum\limits_{q=1}^{Q}{{{\left\| {{\phi }^{2m-1}} \right\|}_{{{L}^{4}}}}{{\left\| {{N}^{q}} \right\|}_{{{L}^{2}}}}}} \notag\\ 
 & \le& c\sum\limits_{m=1}^{M}{\sum\limits_{q=1}^{Q}{\left\| \phi  \right\|_{{{L}^{4\left( 2m-1 \right)}}}^{2m-1}\left\| N \right\|_{{{L}^{2q}}}^{q}}} \notag\\ 
 & \le& c\sum\limits_{m=1}^{M}{\left\| \nabla \phi  \right\|_{{{L}^{2}}}^{\frac{4m-3}{2}}\left\| \phi  \right\|_{{{L}^{2}}}^{\frac{1}{2}}}\sum\limits_{q=1}^{Q}{\left\| \nabla N \right\|_{{{L}^{2}}}^{q-1}}\left\| N \right\|_{{{L}^{2}}}^{{}} \notag\\ 
 & \le&c {\mathcal{J}^{3/2}} ~. 
\end{eqnarray}
The second and the third term of \eqref{doteA0L4} can be estimated as
\begin{eqnarray}
   {{\left\| {{\partial }_{0}}{{F}^{0}} \right\|}_{{{L}^{4/3}}}}&\le& {{\left\| {{\nabla }^{2}}A \right\|}_{{{L}^{4/3}}}}+{{\left\| {{\nabla }^{3}}A \right\|}_{{{L}^{4/3}}}} \notag\\ 
 & \le& \left\| {{\nabla }^{3}}A \right\|_{{{L}^{2}}}^{1/2}\left\| \nabla A \right\|_{{{L}^{2}}}^{1/2}+\left\| {{\nabla }^{4}}A \right\|_{{{L}^{2}}}^{1/2}\left\| {{\nabla }^{2}}A \right\|_{{{L}^{2}}}^{1/2} \notag\\ 
 & \le& \mathcal{J} ~, 
\end{eqnarray}
\begin{eqnarray}
  {{\left\| \phi {{D}_{i}}\phi  \right\|}_{{{L}^{4}}}}& \le&  {{\left\| \phi \nabla \phi  \right\|}_{{{L}^{4}}}}+{{\left\| A{{\phi }^{2}} \right\|}_{{{L}^{4}}}} \notag\\ 
 & \le&  {{\left\| \phi  \right\|}_{{{L}^{8}}}}{{\left\| \nabla \phi  \right\|}_{{{L}^{8}}}}+{{\left\| A \right\|}_{{{L}^{8}}}}{{\left\| {{\phi }^{2}} \right\|}_{{{L}^{8}}}} \notag\\ 
 & \le&  \left\| \nabla \phi  \right\|_{{{L}^{2}}}^{3/4}\left\| \phi  \right\|_{{{L}^{2}}}^{1/4}\left\| {{\nabla }^{2}}\phi  \right\|_{{{L}^{2}}}^{3/4}\left\| \nabla \phi  \right\|_{{{L}^{2}}}^{1/4}+\left\| \nabla A \right\|_{{{L}^{2}}}^{3/4}\left\| A \right\|_{{{L}^{2}}}^{1/4}\left\| \nabla \phi  \right\|_{{{L}^{2}}}^{7/4}\left\| \phi  \right\|_{{{L}^{2}}}^{1/4} \notag\\ 
 & \le&  {\mathcal{J}^{2}}+{\mathcal{J}^{3}} \notag\\
 &\le& {{\left( 1+{{\mathcal{J}}_{0}}+X\left( T \right) \right)}^{3}}~.
\end{eqnarray}
Thus, we have
\begin{eqnarray}
    \int_{0}^{T}{{{\left\| {{\partial }_{t}}{{A}_{0}}\phi  \right\|}_{{{L}^{2}}}}dt}&\le&CT{{\left( 1+{{\mathcal{J}}_{0}}+X\left( T \right) \right)}^{3}}~.
\end{eqnarray}
Next, for the third and fourth term of \eqref{Xineq}, we derive
\begin{eqnarray}
  \int_{0}^{T}{{{\left\| {{A}_{0}}{{\partial }_{t}}\phi  \right\|}_{{{L}^{2}}}}dt}+\int_{0}^{T}{{{\left\| A\nabla \phi  \right\|}_{{{L}^{2}}}}dt}& \le&  C{{T}^{1/2}}\left({{\left\| A\nabla \phi  \right\|}_{{{L}^{2}}\left( \left[ 0,T \right]\times {{\mathbb{R}}^{2}} \right)}}+{{\left\| A_0 
  \partial_t \phi \right\|}_{{{L}^{2}}\left( \left[ 0,T \right]\times {{\mathbb{R}}^{2}} \right)}}\right) \notag\\ 
 & \le&  C{{T}^{1/2}}\left( {{\mathcal{J}}_{0}}+\int_{0}^{T}{\left( {{\left\| \square A \right\|}_{{{L}^{2}}}} \right)dt+\int_{0}^{T}{\left( {{\left\| \square \phi  \right\|}_{{{L}^{2}}}} \right)dt}} \right) \notag\\ 
 & \le&  C{{T}^{1/2}}\left( {{\mathcal{J}}_{0}}+X\left( T \right) \right)  ~.
\end{eqnarray}
One can also derive the estimate for fifth and sixth term of \eqref{Xineq} as follows
\begin{eqnarray}
   \int_{0}^{T}{{{\left\| {{A}^{2}}\phi  \right\|}_{{{L}^{2}}}}dt}+\int_{0}^{T}{{{\left\| {{A_0}^{2}}\phi  \right\|}_{{{L}^{2}}}}dt}&\le& \int_{0}^{T}{\left(\left\| A \right\|_{{{L}^{6}}}^{2}{{\left\| \phi  \right\|}_{{{L}^{6}}}}+\left\| A_0 \right\|_{{{L}^{6}}}^{2}{{\left\| \phi  \right\|}_{{{L}^{6}}}}\right)dt} \notag\\ 
 & \le& CT{{\mathcal{J}}^{2}} \notag\\ 
 & \le & CT{{\left( {{\mathcal{J}}_{0}}+X\left( T \right) \right)}^{2}}~.
\end{eqnarray}
Then, for the Chern-Simmons term, we attain
\begin{eqnarray}
    \int_{0}^{T}{{{\left\|\mathcal{P} {{F}_{i}} \right\|}_{{{L}^{2}}}}dt}\le \int_{0}^{T}{{{\left\| \square A \right\|}_{{{L}^{2}}}}dt}\le X\left( T \right)~.
\end{eqnarray}
Finally, for the potential term of \eqref{Xineq}, we show that
\begin{eqnarray}
   \int_{0}^{T}{{{\left\| {{\partial }_{\phi }}V \right\|}_{{{L}^{2}}}}dt+{{\left\| {{\partial }_{N }}V \right\|}_{{{L}^{2}}}}dt}&\le&\int_{0}^{T}{\left\| \square \phi  \right\|_{{{L}^{2}}}+\left\| \square N  \right\|_{{{L}^{2}}}dt} \notag\\ 
 & \le &X\left( T \right)  
\end{eqnarray}
This concludes the proof of the proposition \ref{prepoANphi}.
\end{proof}
Thus, we can state the following proposition to prove our local existence result:
\begin{preposition}\label{local}
Let $(A_{i},\phi, N)$ be a solution of the Maxwell-Chern-Simons-Higgs equations in a strip $[0,T^*]\times\mathbb{R}^{2}$ in $(2+1)$-dimensional Minkowski spacetime. Define
\begin{eqnarray}
   H(T)&\equiv&\underset{0\le t\le T}{\mathop{\sup }}\,\left( {{\left\| {{A}_{i}}\left( t,\cdot  \right) \right\|}_{{{L}^{2}}\left( {{\mathbb{R}}^{2}} \right)}}+{{\left\| \phi \left( t,\cdot  \right) \right\|}_{{{L}^{2}}\left( {{\mathbb{R}}^{2}} \right)}}+{{\left\| N\left( t,\cdot  \right) \right\|}_{{{L}^{2}}\left( {{\mathbb{R}}^{2}} \right)}} \right. \notag\\ 
 & &+\left. {{\left\| \partial \phi \left( t,\cdot  \right) \right\|}_{{{L}^{2}}\left( {{\mathbb{R}}^{2}} \right)}}+{{\left\| \partial N\left( t,\cdot  \right) \right\|}_{{{L}^{2}}\left( {{\mathbb{R}}^{2}} \right)}} \right)
\end{eqnarray}
Then, there exist constants $C > 0$ and $T^{*} > 0$, depending only on $H(0)$, such that if $T < T^{*}$ then $H(T^*) \le C H(0)$.
\end{preposition}
\begin{proof}
 We define,
\begin{eqnarray}
 X\left( T \right)\equiv\int_{0}^{T}{\left( {{\left\| \square A \right\|}_{{{L}^{2}}}}+{{\left\| \square \phi  \right\|}_{{{L}^{2}}}}+{{\left\| \square N  \right\|}_{{{L}^{2}}}} \right)dt}~.
\end{eqnarray}
The energy estimates show that
\begin{eqnarray}
    H(T^*)\le c \left(H(0)+X(T^*)\right)
\end{eqnarray}
Since by parts (b) preposition \ref{prepoANphi}, $X(T^*)\le C'$, it can be shown that $H(T^*) \le C H(0)$.
\end{proof}

Next, we consider two distinct solutions of $A_0, A, N, \phi$ and $A_0', A', N', \phi' $ that verify \eqref{eomCG2} and \eqref{eomCG3}. 
\begin{preposition}\label{prepoANphiuniqnes}
    Let $A_0, A,  \phi, N$ and  $A_0', A', N', \phi'$ be the two solution of the MCSH system \eqref{eomCG3}, \eqref{eomCG3} and \eqref{eomCG2P} in $[0,T^*]\times\mathbb{R}^2$ with $\mathcal{J}_0$ and  $\mathcal{J'}_0$ finite and satisfies \eqref{i} and \eqref{ii} of the main theorem. Then,
    \begin{enumerate}[label=(\alph*)]
        \item There exists $T\in [0,T^*]$ and a constant $C$ that depend only on $\mathcal{J}_0=\mathcal{J}\left(A,\phi,N\right)(0)$, $\mathcal{J'}_0=\mathcal{J'}\left(A',\phi',N'\right)(0)$  as well as $ X\left( T^* \right)=\int_{0}^{T^*}\left( {{\left\| \square A \right\|}_{{{L}^{2}}}}+{{\left\| \square \phi  \right\|}_{{{L}^{2}}}}+{{\left\| \square N  \right\|}_{{{L}^{2}}}} \right)dt$, $ X'\left( T^* \right)=\int_{0}^{T^*}\left( {{\left\| \square A' \right\|}_{{{L}^{2}}}}+{{\left\| \square \phi'  \right\|}_{{{L}^{2}}}}+{{\left\| \square N'  \right\|}_{{{L}^{2}}}} \right)dt$ such that the following inequality holds:
        \begin{eqnarray}
   \mathcal{M}(T)&=&\int_{0}^{T}{{{\left\| \Box (A-{A}')(t,\cdot ) \right\|}_{{{L}^{2}}}}}dt+\int_{0}^{T}{{{\left\| \Box (\phi -{\phi }')(t,\cdot ) \right\|}_{{{L}^{2}}}}}dt \notag\\ 
 & &+\int_{0}^{T}{{{\left\| \Box (N-{N}')(t,\cdot ) \right\|}_{{{L}^{2}}}}}dt \notag\\ 
 & \le &C\mathcal{J}_0(A - A', \phi - \phi', N-N')~.
        \end{eqnarray}
    \item There exists a real constant $C'$ that depend only on $T^*$, $\mathcal{J}_0$, $\mathcal{J'}_0$, $X\left( T^* \right)$, and $X'\left( T^* \right)$ such that 
    \begin{eqnarray}
       \mathcal{M}(T^*)&=&\int_{0}^{T^*}{{{\left\| \Box (A-{A}')(t,\cdot ) \right\|}_{{{L}^{2}}}}}dt+\int_{0}^{T^*}{{{\left\| \Box (\phi -{\phi }')(t,\cdot ) \right\|}_{{{L}^{2}}}}}dt \notag\\ 
 & &+\int_{0}^{T^*}{{{\left\| \Box (N-{N}')(t,\cdot ) \right\|}_{{{L}^{2}}}}}dt \notag\\ 
 & \le &C\mathcal{J}_0(A - A', \phi - \phi', N-N')~. 
    \end{eqnarray}
    \end{enumerate}
\end{preposition}
\begin{proof}
Part (b) is derived using the conclusion from part (a) in line with the energy estimate \eqref{Mineq}. This process involves iteratively applying the result a finite number of times while taking into account the uniform boundedness of $\mathcal{J}_0\left(t\right)$, $\mathcal{J'}_0\left(t\right)$ in $[0,T^*]\times\mathbb{R}^2$. 

To prove Part (a), we use the same procedure as in the previous preposition, we obtain
\begin{eqnarray}\label{Mineqsuku}
   \mathcal{M}\left( T \right)&\le& \int_{0}^{T}{{{\left\| Im\mathcal{P}\left( \phi \overline{D\phi } \right)-Im\mathcal{P}\left( {\phi }'{{\overline{D\phi }}^{\prime }} \right) \right\|}_{{{L}^{2}}}}dt}+\int_{0}^{T}{{{\left\| \phi \left( {{\partial }_{t}}{{A}_{0}} \right)-{\phi }'\left( {{\partial }_{t}}{{{{A}'}}_{0}} \right) \right\|}_{{{L}^{2}}}}dt} \notag\\ 
 && +\int_{0}^{T}{{{\left\| {{A}_{0}}{{\partial }_{t}}\phi -{{{{A}'}}_{0}}{{\partial }_{t}}{\phi }' \right\|}_{{{L}^{2}}}}dt}+\int_{0}^{T}{{{\left\| A\nabla \phi -{A}'\nabla {\phi }' \right\|}_{{{L}^{2}}}}dt} \notag\\ 
 &&+\int_{0}^{T}{{{\left\| A_{0}^{2}\phi -A{{_{0}^{'}}^{2}}{\phi }' \right\|}_{{{L}^{2}}}}dt}+\int_{0}^{T}{{{\left\| {{A}^{2}}\phi -A{{'}^{2}}{\phi }' \right\|}_{{{L}^{2}}}}dt}+\int_{0}^{T}{{{\left\| \mathcal{P}({{F}_{i}}-{{F}_{i}})^{\prime } \right\|}_{{{L}^{2}}}}dt}\notag\\
 && +\int_{0}^{T}{{{\left\| {{\partial }_{\phi }}V-{{\partial }_{\phi }}{V}' \right\|}_{{{L}^{2}}}}dt}+\int_{0}^{T}{{{\left\| {{\partial }_{N}}V-{{\partial }_{N}}{V}' \right\|}_{{{L}^{2}}}}dt}~.  
\end{eqnarray}
The estimate for each term of \eqref{Mineqsuku} follows directly as in the proof of preposition \ref{prepoANphi}. For the first term of \eqref{Mineqsuku}, we can show that
\begin{eqnarray}
    \int_{0}^{T}{{{\left\| Im\mathcal{P}\left( \phi \text{ }\overline{D\phi } \right)-Im\mathcal{P}\left( {\phi }'\text{ }{{\overline{D\phi }}^{\prime }} \right) \right\|}_{{{L}^{2}}}}dt}\le C{{T}^{1/2}}{{\left( 1+{{\mathcal{J}}_{0}}\left( \phi -{\phi }' \right)+\mathcal{M}\left( T \right) \right)}^{2}}~.
\end{eqnarray}
For the second term of \eqref{Mineqsuku}, we have
\begin{eqnarray}
    \int_{0}^{T}{{{\left\| {{\partial }_{t}}{{A}_{0}}\phi -{{\partial }_{t}}{{{{A}'}}_{0}}{\phi }' \right\|}_{{{L}^{2}}}}dt}\le CT{{\left( 1+{{\mathcal{J}}_{0}}\left( A-{A}',\phi -{\phi }',N-{N}' \right)+\mathcal{M}\left( T \right) \right)}^{3}}
\end{eqnarray}
By applying the same method used in the proof of Preposition \ref{prepoANphi}, after some computation, the remaining terms can be estimated to get 
\begin{eqnarray}\label{Mineq}
\mathcal{M}(T) \leq C T \left(\mathcal{J}_0(A - A', \phi - \phi', N-N') + \mathcal{M}(T)\right).
\end{eqnarray}
with a constant $C$ depending only on $\mathcal{J}_0, \mathcal{X}(T^*), \mathcal{X}'(T^*)$ for all $0 < T \leq T^*$ and $ T \leq 1$. The proof of part (a) follows directly from \eqref{Mineq}. This completes the prof of preposition \ref{prepoANphiuniqnes}.
\end{proof}
To conclude this section, let us consider a solution $A_0, A, \phi, N$ of \eqref{eom3},\eqref{eomCG1},\eqref{eomCG3}, and \eqref{eomCG2P} in $C^\infty([0, T^*) \times \mathbb{R}^2)$, where $A(t, \cdot)$, $\phi(t, \cdot)$, and $N(t, \cdot)$ are compactly supported. We assume $0 < T^* < \infty$ such that by Preposition \ref{prepos1}, $\mathcal{J}(t)$ is uniformly bounded in $[0, T^*)$. Then, we can define the higher energy norm as follows:
\begin{eqnarray}
    \mathcal{J}^{(1)}(t) \equiv \|\partial A(t, \cdot)\|_{H^1(\mathbb{R}^2)} + \|\partial \phi(t, \cdot)\|_{H^1(\mathbb{R}^2)}+\|\partial N(t, \cdot)\|_{H^1(\mathbb{R}^2)}.
\end{eqnarray}
\begin{preposition}\label{prepohihger}
    Suppose that the solution $A_0, A, \phi, N$ of \eqref{eom3},\eqref{eomCG1},\eqref{eomCG3}, and \eqref{eomCG2P}  in $C^\infty([0, T^*) \times \mathbb{R}^2)$ satisfies the assumption made above. Then, $\mathcal{J}^{(1)}(t)$ is uniformly bounded in $[0, T^*)$.
\end{preposition}
\begin{proof}
   Considering the fundamental energy estimate applied to the first derivatives of $A$, $N$, and $\phi$, we have
\begin{eqnarray}
    \|{{\partial }^{2}}A(T,\cdot ){{\|}_{{{L}^{2}}({{\mathbb{R}}^{2}})}}\text{ }\le \|{{\partial }^{2}}A(0,\cdot ){{\|}_{{{L}^{2}}({{\mathbb{R}}^{2}})}}+\int_{0}^{T}{\|\Box \partial A(}t,\cdot ){{\|}_{{{L}^{2}}({{\mathbb{R}}^{2}})}}dt
\end{eqnarray}
\begin{eqnarray}
    \|{{\partial }^{2}}\phi (T,\cdot ){{\|}_{{{L}^{2}}({{\mathbb{R}}^{2}})}}\text{ }\le \|{{\partial }^{2}}\phi (0,\cdot ){{\|}_{{{L}^{2}}({{\mathbb{R}}^{2}})}}+\int_{0}^{T}{\|\Box \partial \phi (}t,\cdot ){{\|}_{{{L}^{2}}({{\mathbb{R}}^{2}})}}dt
\end{eqnarray}
\begin{eqnarray}
    \|{{\partial }^{2}}N(T,\cdot ){{\|}_{{{L}^{2}}({{\mathbb{R}}^{2}})}}\text{ }\le \|{{\partial }^{2}}N(0,\cdot ){{\|}_{{{L}^{2}}({{\mathbb{R}}^{2}})}}+\int_{0}^{T}{\|\Box \partial N(}t,\cdot ){{\|}_{{{L}^{2}}({{\mathbb{R}}^{2}})}}dt
\end{eqnarray}
Therefore, it suffices to prove the following corollary of Preposition \ref{prepoANphiuniqnes}.
\end{proof}
\begin{preposition}\label{prep7}
    There exists a positive constant $ C$, depend only on $\mathcal{J}_0$ and $ T^*$, such that
\begin{eqnarray}
    \int_{0}^{{{T}^{*}}}{\left( \|\Box \partial A{{\|}_{{{L}^{2}}}}+\|\Box \partial \phi {{\|}_{{{L}^{2}}}}+\|\Box \partial N{{\|}_{{{L}^{2}}}} \right)}dt\le C{{\mathcal{J}}^{(1)}}(0).
\end{eqnarray} 
\end{preposition}
\begin{proof}
    The estimate is derived from part (b) of Preposition \ref{prepoANphiuniqnes}, which is applied to finite difference approximations of partial derivatives.
\end{proof}
Once an estimate for $\mathcal{J}^{(1)}$ is obtained, a subsequent estimate for $\mathcal{J}^{(s)}$ can be given by the following proposition:
\begin{preposition}\label{prep9}
For $s>0$, we define
    \begin{eqnarray}
    \mathcal{J}^{(s)}(t) \equiv \|\partial A(t, \cdot)\|_{H^s(\mathbb{R}^2)} + \|\partial \phi(t, \cdot)\|_{H^s(\mathbb{R}^2)}+\|\partial N(t, \cdot)\|_{H^s(\mathbb{R}^2)}.
\end{eqnarray}
Then, $\mathcal{J}^{(s)}(t)$ is uniformly bounded in $[0, T^*)$.
\end{preposition}
\begin{proof}
    If we also consider the fundamental energy estimate that is applied to the s-th derivatives of $A$, $N$ and $\phi$, then we have
\begin{eqnarray}
    \|{{\partial }^{s}}A(T,\cdot ){{\|}_{{{L}^{2}}({{\mathbb{R}}^{2}})}}\text{ }\le \|{{\partial }^{s}}A(0,\cdot ){{\|}_{{{L}^{2}}({{\mathbb{R}}^{2}})}}+\int_{0}^{T}{\|\Box \partial^{s-1} A(}t,\cdot ){{\|}_{{{L}^{2}}({{\mathbb{R}}^{2}})}}dt
\end{eqnarray}
\begin{eqnarray}
    \|{{\partial }^{s}}\phi (T,\cdot ){{\|}_{{{L}^{2}}({{\mathbb{R}}^{2}})}}\text{ }\le \|{{\partial }^{s}}\phi (0,\cdot ){{\|}_{{{L}^{2}}({{\mathbb{R}}^{2}})}}+\int_{0}^{T}{\|\Box \partial^{s-1} \phi (}t,\cdot ){{\|}_{{{L}^{2}}({{\mathbb{R}}^{2}})}}dt
\end{eqnarray}
\begin{eqnarray}
    \|{{\partial }^{s}}N(T,\cdot ){{\|}_{{{L}^{2}}({{\mathbb{R}}^{2}})}}\text{ }\le \|{{\partial }^{s}}N(0,\cdot ){{\|}_{{{L}^{2}}({{\mathbb{R}}^{2}})}}+\int_{0}^{T}{\|\Box \partial^{s-1} N(}t,\cdot ){{\|}_{{{L}^{2}}({{\mathbb{R}}^{2}})}}dt
\end{eqnarray} 
Since we have preposition \ref{prepohihger} and \ref{prep7}, by using a iteration method along with corollary of Preposition \ref{prepoANphiuniqnes}, it can be shown that $\mathcal{J}^{(s)}(t)$ is uniformly bounded in $[0, T^*)$.
\end{proof}
\begin{preposition}\label{prep10}
    For $s>0$, there exists a constant $C$, depend only on $\mathcal{J}_0$ and $ T^*$, such that
\begin{eqnarray}
    \int_{0}^{{{T}^{*}}}{\left( \|\Box \partial^{s-1} A{{\|}_{{{L}^{2}}}}+\|\Box \partial^{s-1} \phi {{\|}_{{{L}^{2}}}}+\|\Box \partial^{s-1} N{{\|}_{{{L}^{2}}}} \right)}dt\le C{{\mathcal{J}}^{(s-1)}}(0).
\end{eqnarray} 
\end{preposition}
\begin{proof}
The proof can be derived by using an iteration method along with part (b) of Preposition \ref{prepoANphiuniqnes}.
\end{proof}

\section{Proof of the main theorem}\label{sec:MCSH prove}
\begin{enumerate}
    \item \textbf{Uniqueness}: Let $A_0, A, \phi, N, A_0', A', \phi', N'$ be two solutions to \eqref{eom3}, \eqref{eomCG1},\eqref{eomCG3}, and \eqref{eomCG2P} in a slab $[0, T^*] \times \mathbb{R}^2$ associated with the same initial data as in \eqref{initialdata} . Thus $\mathcal{J}_0$ is finite. Suppose these solutions also fulfill the conditions \eqref{i}, \eqref{ii} 
of the main theorem. Then $A_0 = A_0'$, $A = A'$, $\phi = \phi'$ and $N=N'$ in $[0, T^*] \times \mathbb{R}^2$.
\begin{proof}
    The equalities $ A = A' $, $ \phi = \phi' $, $ N = N' $ follow immediately from the Preposition \ref{prepoANphiuniqnes}. In view of Preposition \ref{prepo1}, we then find $ A_0 = A_0'$. 
\end{proof}

\item \textbf{Global existence for compactly smooth supported data}: Consider the initial data given in \eqref{initialdata} for the system of equations \eqref{eom3}, \eqref{eomCG1},\eqref{eomCG3}, and \eqref{eomCG2P}. Let $T^*$ denote the maximal of $T$ such that a $C^\infty$ solution exists on $[0, T^*] \times \mathbb{R}^2$. Then, it holds that $T^* = \infty$.
\begin{proof}
    Suppose, by contradiction, $T^* < \infty$. From the classical local existence result in preposition \ref{local}, it follows that $T^* > 0$. Furthermore, by Proposition \ref{prepohihger}, $\mathcal{J}^{(1)}(t)$ remains uniformly bounded on $[0, T^*)$. Applying Proposition \ref{local} once more and also using the result in preposition \ref{prep7}, the solution can be extended beyond $T^*$, leading to a contradiction.
\end{proof}
\item \textbf{Global existence for data with finite $\mathcal{J}_0$}: Let the initial data be given as in \eqref{initialdata} with $\mathcal{J}_0$ finite. Then, the system of equations \eqref{eom3}, \eqref{eomCG1},\eqref{eomCG3}, and \eqref{eomCG2P} admits a global solution that satisfies conditions \eqref{i} and \eqref{ii} of the main theorem.
\begin{proof}
    Let construct a sequence of data $(a_0^{(m)}(0, \cdot), a_1^{(m)}(0, \cdot), \varphi_0^{(m)}(0, \cdot), \varphi_1^{(m)}(0, \cdot))$,  $n_0^{(m)}(0,\cdot)$  $n_1^{(m)}(0, \cdot) \in C^\infty_c(\mathbb{R}^2)$ with $\text{div}\, a_0^{(m)}(0) = 0$, $\text{div}\, a_1^{(m)}(0) = 0$ such that
    \begin{eqnarray}
        \|a_0^{(m)}(0, \cdot) - a_0(0, \cdot)\|_{H^1} &\to& 0,\\
        \|\varphi_0^{(m)}(0, \cdot) - \varphi_0(0, \cdot)\|_{H^1} &\to& 0,\\
         \|n_0^{(m)}(0, \cdot) - n_0(0, \cdot)\|_{H^1} &\to& 0,\\
        \|a_1^{(m)}(0, \cdot) - a_1(0, \cdot)\|_{L^2} &\to& 0,\\
        \|\varphi_1^{(m)}(0, \cdot) - \varphi_1(0, \cdot)\|_{L^2} &\to& 0,\\
        \|n_1^{(m)}(0, \cdot) - n_1(0, \cdot)\|_{L^2} &\to& 0.
    \end{eqnarray}
Let $(a^{(m)}, n^{(m)}, \varphi^{(m)})$ denote the global solutions obtained from the previous result. By Propositions \ref{prepoANphi}, \ref{prepoANphiuniqnes}, and \ref{prepo1}, this sequence converges to a limit $(A_0, A, N, \varphi)$ that satisfies conditions \eqref{i} and \eqref{ii} of the main theorem.
\end{proof}
\item \textbf{Global existence for higher regularity $H^s$ data}: Suppose that the initial data satisfy the higher regularity assumptions 
\begin{eqnarray}
   &&\nabla {{a}_{\left( 0 \right)}}\in {{H}^{s}}\left( {{\mathbb{R}}^{2}} \right), {{a}_{\left( 1 \right)}}\in {{H}^{s}}\left( {{\mathbb{R}}^{2}} \right),{{\varphi }_{\left( 0 \right)}}\in {{H}^{s+1}}\left( {{\mathbb{R}}^{2}} \right),\notag\\
   &&{{\varphi }_{\left( 1 \right)}}\in {{H}^{s}}\left( {{\mathbb{R}}^{2}} \right),\nabla {{n}_{\left( 0 \right)}}\in {{H}^{s}}\left( {{\mathbb{R}}^{2}} \right), {{n}_{\left( 1 \right)}}\in {{H}^{s}}\left( {{\mathbb{R}}^{2}} \right)~, 
\end{eqnarray}
for some integer $s>0$, then,
\begin{eqnarray}
   A_0, A&\in& {{H}^{s+1}}\left( {{\mathbb{R}}^{2}} \right)~,\notag\\
   {{\partial }_{t}}A_0, {{\partial }_{t}}A&\in& {{H}^{s}}\left( {{\mathbb{R}}^{2}} \right),\notag\\
   \phi,N &\in& {{H}^{s+1}}\left( {{\mathbb{R}}^{2}} \right)~,\notag\\
   {{\partial }_{t}}\phi,\partial_t N &\in& {{H}^{s}}\left( {{\mathbb{R}}^{2}} \right). 
\end{eqnarray}
\begin{proof}
By contradiction again, let assume that $T^* < \infty$. Since by Proposition \ref{prep9}, $\mathcal{J}^{(s)}(t)$ remains uniformly bounded on $[0, T^*)$. Applying Proposition \ref{local} once more along with preposition \ref{prep10}, the solution can be extended beyond $T^*$, this leads to a contradiction.
\end{proof}
\end{enumerate}
\section*{Acknowledgements}
This work is supported by Postdoctoral Research Program BRIN 2024. F. T. A. and B. E. G. also acknowledge ITB Research Grant for financial support.

\section*{Data availibility statement }
The manuscript has no associated data.

\section*{Conflict of interest }
This work does not have any conflicts of interest.

\end{document}